\documentclass[12pt]{amsart}

\usepackage{amssymb, amscd, txfonts}
\usepackage[pdftex]{graphicx}
\usepackage[all]{xy}

\numberwithin{equation}{section}

\newtheorem{theorem}{Theorem}[section]
\newtheorem{proposition}[theorem]{Proposition}
\newtheorem{lemma}[theorem]{Lemma}
\newtheorem{corollary}[theorem]{Corollary}

\theoremstyle{definition}

\newtheorem{condition}[theorem]{Condition}

\theoremstyle{remark}
\newtheorem{remark}[theorem]{Remark}

\renewcommand{\ker}{\operatorname{Ker}}

\newcommand{\Z}{\mathbb{Z}}
\newcommand{\Q}{\mathbb{Q}}
\newcommand{\R}{\mathbb{R}}
\newcommand{\C}{\mathbb{C}}

\newcommand{\proj}{{\mathbb P}}

\newcommand{\D}{\mathcal{D}}
\newcommand{\G}{\Gamma}
\newcommand{\Gr}{\rm Gr}
\newcommand{\LL}{\mathcal{L}}
\newcommand{\GFZ}{\Gamma(F)}
\newcommand{\GF}{\Gamma_{F}}
\newcommand{\UFZ}{U(F)_{\mathbb{Z}}}
\newcommand{\GFZbar}{\overline{\Gamma(F)}}
\newcommand{\GFZdash}{\Gamma(F)'}
\newcommand{\GFZbd}{\overline{\Gamma(F)'}}
\newcommand{\VF}{\mathcal{V}_{F}}
\newcommand{\XF}{\mathcal{X}(F)} 
\newcommand{\XFcpt}{\mathcal{X}(F)^{\Sigma_{F}}} 
\newcommand{\YFcpt}{\mathcal{Y}(F)^{\Sigma_{F}}} 
\newcommand{\Xcpt}{X^{\Sigma}}
\newcommand{\SFs}{S\!(F,\sigma)}
\newcommand{\DFs}{D(F,\sigma)}
\newcommand{\DFsbar}{\overline{D(F,\sigma)}}
\newcommand{\Mi}{M_{can}^{(i)}}
\newcommand{\res}{{\rm res}}

\DeclareMathOperator{\aut}{Aut}
\DeclareMathOperator{\coker}{Coker}

\begin{document}

\title[]{Mixed Hodge structures and Siegel operators}
\author[]{Shouhei Ma}
\thanks{Supported by KAKENHI 21H00971 and 20H00112} 
\address{Department~of~Mathematics, Tokyo~Institute~of~Technology, Tokyo 152-8551, Japan}
\email{ma@math.titech.ac.jp}

\begin{abstract}
In this paper we study mixed Hodge structures on the cohomology of locally symmetric varieties 
and give an application to modular forms. 
After proving vanishing of some Hodge numbers, 
we focus on the weight filtration on the last Hodge subspace of the middle degree cohomology. 
We prove that the weight filtration coincides with 
the corank filtration on the space of modular forms of canonical weight defined by the Siegel operators, 
and calculate the graded quotients. 
As an application, we deduce surjectivity of the total Siegel operators in many cases, 
and identify an obstruction space in the remaining case. 
\end{abstract}

\maketitle

\section{Introduction}\label{sec: intro}

Cohomology of locally symmetric varieties is a rich object of study 
which lies at the crossroads of various branches of Mathematics such as  
Harmonic Analysis, Number theory and Algebraic Geometry. 
In this paper we study mixed Hodge structures on them. 
We connect the weight filtration with 
the Siegel operators for modular forms of canonical weight. 
This in turn has an application to a classical problem about Siegel operators, 
which is a new instance of application of mixed Hodge theory. 

Let ${\D}$ be a Hermitian symmetric domain and ${\G}$ be an arithmetic subgroup of 
$G={\rm Aut}({\D})^{\circ}=\mathbb{G}({\R})^{\circ}$, 
where $\mathbb{G}$ is a connected semi-simple linear algebraic group over ${\Q}$. 
For simplicity of exposition we assume that ${\G}$ is neat in this introduction. 
(See \S \ref{sec: non-neat} for the non-neat case.)  
The quotient $X={\D}/{\G}$ has the structure of a smooth quasi-projective variety (\cite{BB}), 
known under various names such as 
\textit{locally symmetric variety} or \textit{arithmetic quotient} or \textit{modular variety} or \textit{Shimura variety}. 
By the theory of Deligne \cite{DeII}, the cohomology 
$H^k(X)=H^k(X, {\C})$ of $X$ has a canonical mixed Hodge structure $(W_{\bullet}, F^{\bullet})$, 
where $W_{\bullet}$ is the weight filtration and $F^{\bullet}$ is the Hodge filtration. 
The weight graded quotient ${\Gr}_{m}^{W}H^k(X)$ is nonzero only in the range 
$k\leq m \leq \min (2k, 2n)$ where $n=\dim X$. 
The lowest part ${\Gr}^{W}_{k} H^k(X)=W_k H^k(X)$ 
is the image of the natural map from the $L^2$-cohomology (\cite{HZIII}) 
and hence related to automorphic forms in the discrete spectrum. 
On the other hand, the higher graded quotients ${\Gr}^{W}_{m} H^k(X)$, $m>k$, 
also have geometric significance. 
They should be related to automorphic forms in the continuous spectrum, 
but currently the relation seems mysterious (cf.~\cite{OS1}, \cite{Na}). 
In this paper we study the higher weight graded quotients by using toroidal compactifications of $X$. 

In order to simplify the exposition, 
we assume in this introduction that $\mathbb{G}$ is ${\Q}$-simple and 
concentrate on the middle degree $k=n>1$, 
where usually the cohomology is most interesting. 
Let $r$ be the ${\Q}$-rank of $\mathbb{G}$. 
Then we have a flag 
$F_1 \succ \cdots \succ F_r$ of reference cusps of ${\D}$ of length $r$ such that  
every cusp of ${\D}$ is equivalent to one of $F_{i}$ under the action of $\mathbb{G}({\Q})$. 
The subscript $i$ of $F_{i}$ is referred to as the \textit{corank} of $F_{i}$. 
For $1\leq i \leq r$ we denote by $n(i)$ the dimension of the center of 
the unipotent radical of the stabilizer of $F_{i}$ in $G$. 
Then $0<n(1)< \cdots <n(r)\leq n$. 
We have $n(r)=n$ if and only if ${\D}$ has a tube domain realization. 
We also set $n(0)=0$.  

As a preliminary for our main result, we first prove vanishing of some Hodge numbers. 
Let $h^{p,q}_{n}(X)=\dim {\Gr}_{F}^{p}{\Gr}_{p+q}^{W}H^n(X)$ be the Hodge number of $H^n(X)$ of degree $(p, q)$. 
Since $X$ is smooth, we have $h_{n}^{p,q}(X)\ne 0$ only when $p+q\geq n$, $p\leq n$, $q\leq n$ (\cite{DeII}). 
We derive the following further constraint (Theorem \ref{thm: VT I}). 

\begin{theorem}\label{thm: VT intro}
Let $n(i-1)<m \leq n(i)$ with $1\leq i \leq r$. 
Then ${\Gr}_{n+m}^WH^n(X)$ is a Tate twist of 
an effective pure Hodge structure of weight $\leq n-n(i)$. 
Therefore, if $n(i-1)<p+q-n\leq n(i)$, we have $h^{p,q}_{n}(X)\ne 0$ only when $|p-q|\leq n-n(i)$. 
When $m>n(r)$, we have ${\Gr}^{W}_{n+m}H^n(X)=0$. 
\end{theorem}

See Figure \ref{figure: Hodge terrace} for a (rotated) shape of the range of $(p, q)$ 
described in Theorem \ref{thm: VT intro}, 
which looks like stairs with steps in heights $n(1), \cdots, n(r)$. 
Here the bottom line is $p+q=n$, and the dotted roof is $p=n$ and $q=n$. 

\begin{figure}[h]\label{figure: Hodge terrace}
\includegraphics[height=36mm, width=72mm]{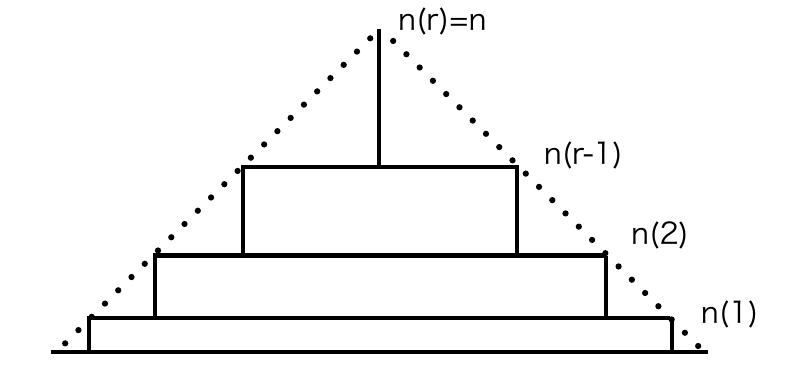}
\caption{Hodge stairs (the case $n=n(r)$)}
\end{figure}

A similar result holds for general degree $k$ (see Theorem \ref{thm: VT I}). 
Moreover, when $k<n$, the edge Hodge numbers 
$h^{k,m}_k(X)$ with $m>0$ vanish (Proposition \ref{thm: VT II}), 
which generalizes a result of Oda-Schwermer \cite{OS1}. 

Our main focus is on the weight filtration $W_{\bullet}F^n=W_{\bullet}H^n(X) \cap F^nH^n(X)$ 
on the subspace $F^n=F^nH^n(X)$. 
(This is the right dotted roof in Figure \ref{figure: Hodge terrace}.)
Theorem \ref{thm: VT intro} implies that this filtration reduces to length $r$: 
\begin{equation}\label{eqn: weight filtration intro}
W_nF^n \subset W_{n+n(1)}F^n \subset W_{n+n(2)}F^n \subset \cdots \subset W_{n+n(r)}F^n =F^n. 
\end{equation}
On the other hand, by a theorem of Mumford \cite{Mu}, 
the space $F^n$ is naturally isomorphic to the space 
$M_{can}=M_{can}({\G})$ of ${\G}$-modular forms of canonical weight on ${\D}$, 
i.e., ${\G}$-invariant canonical forms on ${\D}$. 
In typical examples, the canonical weight is 
weight $g+1$ for ${\rm Sp}(2g, {\R})$, 
weight $n$ for ${\rm O}(2, n)$, and  
weight $p+q$ for ${\rm U}(p, q)$, 
while when working in the full generality of locally symmetric varieties, 
it is the smallest common weight.
We have the \textit{corank filtration} on $M_{can}$: 
\begin{equation}\label{eqn: corank filtration intro}
M_{can}^{(0)} \subset M_{can}^{(1)} \subset M_{can}^{(2)} \subset \cdots \subset M_{can}^{(r)} = M_{can}, 
\end{equation}
where $M_{can}^{(i)}\subset M_{can}$ is the subspace of modular forms 
which vanish at all cusps of ${\D}$ of corank $>i$. 
In particular, $M_{can}^{(0)}$ is the space of cusp forms. 
For each cusp $F$ of corank $i$, we have the (cuspidal) Siegel operator 
\begin{equation*}
\Phi_{F} : M_{can}^{(i)} \to S\!_{cat}(F)
\end{equation*}
as the restriction operator to $F$, 
where $S\!_{cat}(F)$ is the space of cusp forms on $F$ of suitable weight, say $cat$, 
with respect to the integral stabilizer ${\GFZ}<{\G}$ of $F$. 
Let $\mathcal{C}(i)$ be the set of ${\G}$-equivalence classes of cusps of ${\D}$ of corank $i$. 
By taking the direct sum of $\Phi_{F}$ over all $F\in \mathcal{C}(i)$, 
we obtain the total cuspidal Siegel operator of corank $i$: 
\begin{equation*}
\Phi^{(i)}_{0} : M_{can}^{(i)} \to \bigoplus_{F\in \mathcal{C}(i)}S\!_{cat}(F). 
\end{equation*}

Our main result is comparison of the two filtrations 
\eqref{eqn: weight filtration intro}, \eqref{eqn: corank filtration intro} 
and calculation of the graded quotients 
(Theorem \ref{thm: weight=corank} and Theorem \ref{thm: CD}). 

\begin{theorem}\label{thm: main intro}
Under the natural isomorphism $F^nH^n(X)\simeq M_{can}({\G})$, 
the weight filtration \eqref{eqn: weight filtration intro} coincides with the corank filtration \eqref{eqn: corank filtration intro}. 
When $i>1$ or $i=1$ with $n(1)>1$, 
we have the following commutative diagram 
where all maps are isomorphisms: 
\begin{equation*}
\xymatrix{
W_{n+n(i)}F^n/W_{n+n(i-1)}F^n \ar[r]^-{{\res}} & \bigoplus_{F\in \mathcal{C}(i)}H^0(K_{\overline{Y}_{F}})^{{\GF}}  \\ 
M_{can}^{(i)}/M_{can}^{(i-1)} \ar[u]  \ar[r]^{\Phi^{(i)}_{0}} & \bigoplus_{F\in \mathcal{C}(i)}S\!_{cat}(F) \ar[u] 
}
\end{equation*}
\end{theorem}
 
Here 
$\overline{Y}_{F}$ is a smooth projective compactification of a certain abelian fibration $Y_{F}$ attached to $F$, 
${\GF}$ is the non-Hermitian part of the Levi part of ${\GFZ}$, 
${\res}$ is a map induced by the residue maps, 
and the right vertical map is a collection of certain Shioda-type maps (cf.~\cite{Shioda} for a prototype). 


Three key points of the proof of Theorem \ref{thm: main intro} are 
construction of the Shioda-type maps $S\!_{cat}(F) \to H^0(K_{\overline{Y}_{F}})$ 
which connect the Siegel operators with the residue maps, 
the ${\GF}$-invariance of the Shioda-type maps, 
and independent calculation of ${\Gr}^{W}_{n+n(i)}F^n$. 
Comparison of the combinatorial structures in toroidal compactifications and mixed Hodge theory is essential.  

Theorem \ref{thm: main intro} has the following consequence on the Siegel operators. 

\begin{corollary}\label{cor: intro}
When $i>1$ or $i=1$ with $n(1)>1$, 
the cuspidal total Siegel operator $\Phi^{(i)}_{0}$ of corank $i$ is surjective. 
\end{corollary}

The remaining case $i=1$ with $n(1)=1$ is indeed an exception (cf.~\cite{SM}).  
In this case, we show that $H^0(X, \Omega_{X}^{n-1})$ provides 
an obstruction space for $\Phi^{(1)}_{0}$ to be surjective 
(Proposition \ref{prop: cokernel Siegel n(1)=1}). 

A standard way to study surjectivity of Siegel operators is to construct Eisenstein series. 
However, Eisenstein series usually do not converge in the canonical weight, 
and a traditional approach in this situation is to use its analytic continuation (the so-called Hecke summation). 
Shimura \cite{Shimura} developed this method for $0$-dimensional cusps for 
$G={\rm Sp}(2g, {\R})$ and ${\rm SU}(p, p)$, 
and Weissauer \cite{We} obtained a best result for cusps of arbitrary corank for ${\G}={\rm Sp}(2g, {\Z})$. 
In these cases, Corollary \ref{cor: intro} is covered by their results. 
Our method gives a new approach, based on mixed Hodge theory and not using Eisenstein series. 
Our result is for a particular weight, but applies uniformly to any locally symmetric variety. 

Mixed Hodge structures of locally symmetric varieties have been studied by 
Oda \cite{Od} and Ziegler \cite{Fr} for Hilbert modular varieties, 
Oda-Schwermer \cite{OS1}, \cite{OS2} and Hoffman-Weintraub \cite{HW1}, \cite{HW2} 
for Siegel modular $3$-folds and their inner forms, 
and Miyazaki \cite{Mi} for Siegel modular varieties. 
Oda and Miyazaki used Eisenstein series (by Hecke summation)  
to study the contribution from the $0$-dimensional cusps to the top graded piece ${\Gr}^{W}_{2n}H^n(X)$. 
Their works were the first sign of the relation between the weight filtration and the Siegel operators. 
On the other hand, 
Ziegler and Hoffman-Weintraub calculated ${\Gr}^{W}_{2n}H^n(X)$ via simplicial homology. 
Our approach is a mixture of these two methods, 
generalized to reveal the connection between lower graded quotients and cusps of positive dimension 
by using the Shioda-type maps. 
The use of Eisenstein series is bypassed: in a sense, the order is reversed. 

As further related works, 
we also mention the papers of 
Harris-Zucker \cite{HZII}, Zucker \cite{Zu}, Nair \cite{Na} and Odaka-Oshima \cite{OO} 
who studied the mixed Hodge structures of locally symmetric varieties  
from the viewpoint of (reductive) Borel-Serre compactifications. 
 
The rest of this paper is organized as follows. 
In \S \ref{sec: MHS}, we recall mixed Hodge theory. 
In \S \ref{sec: toroidal}, we recall the theory of toroidal compactifications. 
In \S \ref{sec: VT I}, we prove the first vanishing theorem (Theorem \ref{thm: VT I}).  
In \S \ref{sec: VT II}, we prove the second vanishing theorem (Proposition \ref{thm: VT II}). 
In \S \ref{sec: weight Siegel}, we prove the coincidence  
\eqref{eqn: weight filtration intro} = \eqref{eqn: corank filtration intro} of the filtrations, 
and also prove the ${\GF}$-invariance of the Shioda-type maps. 
In \S \ref{sec: GrW}, we calculate ${\Gr}^{W}_{n+n(i)}F^n$. 
In \S \ref{sec: non-neat}, we extend some of our main results to non-neat ${\G}$. 
In \S \ref{sec: example}, we explain our results more explicitly for some classical domains ${\D}$. 
In Appendix \ref{sec: SNC}, we give a proof of the existence of a toroidal compactification with SNC boundary divisor, 
a folklore result whose proof is scattered or only implicit in the literature. 


\section{Mixed Hodge theory}\label{sec: MHS}

In this section we recall mixed Hodge theory following \cite{DeII}, \cite{DeIII}, \cite{PS}. 

\subsection{Mixed Hodge structures} 

A (${\Q}$-)mixed Hodge structure is a triplet $(V, W_{\bullet}, F^{\bullet})$ 
where $V$ is a ${\Q}$-linear space of finite dimension, 
$W_{\bullet}$ is an increasing filtration on $V$ (the weight filtration), 
$F^{\bullet}$ is a decreasing filtration on $V_{{\C}}$ (the Hodge filtration), 
such that the filtration induced by $F^{\bullet}$ on 
${\Gr}_{m}^{W}V_{{\C}}=W_{m}V_{{\C}}/W_{m-1}V_{{\C}}$ 
defines a pure Hodge structure of weight $m$. 
The Hodge numbers of $(V, W_{\bullet}, F^{\bullet})$ are defined by 
\begin{equation*}
h^{p,q}(V) = \dim {\Gr}_{F}^{p}{\Gr}_{p+q}^{W} V_{{\C}}. 
\end{equation*}
By Deligne \cite{DeII}, \cite{DeIII}, the cohomology $H^k(X, {\Q})$ of a complex algebraic variety $X$ 
has a canonical mixed Hodge structure. 
In this paper we will be mainly interested in the case when $X$ is smooth. 
Let us recall the construction of the filtrations $(W_{\bullet}, F^{\bullet})$ in this case.  

Let $X$ be a smooth complex algebraic variety of dimension $n$. 
We take a smooth algebraic compactification $X\hookrightarrow \bar{X}$ such that 
the boundary $D=\bar{X}-X$ is a divisor with simple normal crossings. 
Let $\Omega_{\bar{X}}^{\bullet}(\log D)$ be the logarithmic de Rham complex of $(\bar{X}, D)$. 
Then we have a natural isomorphism 
$H^{k}(X, {\C}) \simeq \mathbb{H}^k(\bar{X}, \Omega_{\bar{X}}^{\bullet}(\log D))$. 
The weight filtration on the complex $\Omega_{\bar{X}}^{\bullet}(\log D)$ is defined by 
\begin{equation*}
W_m \Omega_{\bar{X}}^{p}(\log D) = 
\begin{cases}
\; 0 & \; m<0 \\
\; \Omega_{\bar{X}}^{p}(\log D) & \; m\geq p \\
\; \Omega_{\bar{X}}^{p-m}\wedge \Omega_{\bar{X}}^{m}(\log D) & \; 0\leq m \leq p
\end{cases}
\end{equation*}
We also let $(F^{p}\Omega_{\bar{X}}^{\bullet}(\log D))_{p}$ be the trivial decreasing filtration on 
the complex $\Omega_{\bar{X}}^{\bullet}(\log D)$. 
Then the weight filtration on $H^k(X)=H^k(X, {\C})$ is defined by 
\begin{equation*}
W_{k+m}H^k(X) = 
{\rm Im}(\mathbb{H}^k(\bar{X}, W_{m}\Omega_{\bar{X}}^{\bullet}(\log D)) 
\to \mathbb{H}^k(\bar{X}, \Omega_{\bar{X}}^{\bullet}(\log D))), 
\end{equation*} 
and the Hodge filtration is defined by 
\begin{equation*}
F^pH^k(X) = 
{\rm Im}(\mathbb{H}^k(\bar{X}, F^{p}\Omega_{\bar{X}}^{\bullet}(\log D)) 
\to \mathbb{H}^k(\bar{X}, \Omega_{\bar{X}}^{\bullet}(\log D))). 
\end{equation*} 
We have the weight spectral sequence 
\begin{equation*}
\!_{W}E_{1}^{p,q} = \mathbb{H}^{p+q}(\bar{X}, {\Gr}^{W}_{-p}\Omega_{\bar{X}}^{\bullet}(\log D)) 
\; \; \Rightarrow \; \; 
\!_{W}E_{\infty}^{p+q} = H^{p+q}(X) 
\end{equation*}
which is a second-quadrant spectral sequence converging to the weight filtration. 
We also have the Hodge spectral sequence 
\begin{equation*}
\!_{F}E_{1}^{p,q} = H^{q}(\bar{X}, \Omega_{\bar{X}}^{p}(\log D)) 
\; \; \Rightarrow \; \; 
\!_{F}E_{\infty}^{p+q} = H^{p+q}(X) 
\end{equation*}
which is a first-quadrant spectral sequence converging to the Hodge filtration. 
It is a fundamental theorem of Deligne \cite{DeII} that 
the weight spectral sequence degenerates at the $E_{2}$ page, 
the Hodge spectral sequence degenerates at the $E_{1}$ page, 
and $(W_{\bullet}, F^{\bullet})$ defines a mixed Hodge structure on $H^{k}(X, {\Q})$.

\subsection{Weight spectral sequence}\label{ssec: weight ss}

We recall the explicit description of 
the $E_1$ to $E_2$ pages of the weight spectral sequence following \cite{DeII} and \cite{PS} Chapter 4. 
In what follows, we abbreviate $\!_{W}E_{r}^{p,q}=E_{r}^{p,q}$. 
It is sometimes convenient to switch from the index $(p, q)$ to $(k, m)$ by 
\begin{equation*}\label{eqn: (p,q) (m,k)} 
(k, m) = (p+q, -p) \: \Leftrightarrow \: (p, q)=(-m, k+m). 
\end{equation*}

We denote by $D=\sum_{i}D_{i}$ the irreducible decomposition of $D$, 
where the irreducible components are numbered. 
Let $0\leq m \leq n$. 
For an index set $I=\{ i_1< \cdots <i_m \}$ with $|I|=m$, we write 
\begin{equation*}
D_{I} = D_{i_1}\cap \cdots \cap D_{i_m}. 
\end{equation*}
By the SNC condition, $D_{I}$ is a smooth subvariety of $\bar{X}$ of codimension $m$ 
if it is not empty. 
We denote by $D_{I}'=\sum_{i\not\in I} D_{i}|_{D_{I}}$ the boundary divisor of $D_{I}$ itself. 
We have the residue map 
\begin{equation*}
{\res}_{I} : \Omega_{\bar{X}}^{p}(\log D)\to \Omega_{D_{I}}^{p-m}(\log D_{I}'), 
\end{equation*}
which is locally defined by 
\begin{equation*}
\frac{dz_1}{z_1}\wedge \cdots \wedge \frac{dz_m}{z_m} \wedge \eta +\eta'   \mapsto (2\pi i)^m\eta |_{D_{I}}, 
\end{equation*}
where $z_1, \cdots, z_m$ are local equations of $D_{i_1}, \cdots, D_{i_m}$. 
On the subsheaf $W_{m}\Omega_{\bar{X}}^{p}(\log D)$, this restricts to 
\begin{equation*}
{\res}_{I} : W_{m}\Omega_{\bar{X}}^{p}(\log D)\to \Omega_{D_{I}}^{p-m}. 
\end{equation*}
Let 
\begin{equation}\label{eqn: D(m)}
D(m) = \bigsqcup_{|I|=m}D_{I} 
\end{equation}
and $a_m\colon D(m)\to \bar{X}$ be the natural map. 
Then the residue maps ${\res}_{I}$ with $|I|=m$ induce an isomorphism 
\begin{equation*}\label{eqn: resm}
{\res}_{m} : {\Gr}_{m}^{W} \Omega_{\bar{X}}^{\bullet}(\log D) 
\stackrel{\simeq}{\to} (a_{m})_{\ast}\Omega_{D(m)}^{\bullet}[-m] 
\end{equation*}
of complexes of sheaves on $\bar{X}$. 
This defines an isomorphism 
\begin{equation}\label{eqn: weight E1}
E_{1}^{-m,k+m} \simeq H^{k-m}(D(m))(-m).  
\end{equation}
Here $V(-m)$ stands for the $(-m)$-th Tate twist of a pure Hodge structure $V$. 

By \eqref{eqn: weight E1}, the differential 
\begin{equation*}
d_1\colon E_{1}^{-m,k+m}\to E_{1}^{-m+1,k+m} 
\end{equation*}
in the weight spectral sequence is a map of the form 
\begin{equation*} 
d_1 : H^{k-m}(D(m))(-m) \to H^{k-m+2}(D(m-1))(-m+1). 
\end{equation*}  
On each component $H^{k-m}(D_{I})$ of $H^{k-m}(D(m))$ where $I=\{ i_1< \cdots <i_m \}$, 
the map $-d_{1}$ is identified with 
\begin{equation}\label{eqn: Gysin sign}
\bigoplus_{j=1}^{m}(-1)^{j-1} \rho^{I}_{j} \: : \: 
H^{k-m}(D_{I}) \to \bigoplus_{j=1}^{m} H^{k-m+2}(D_{I \backslash \{ i_{j} \}})(1), 
\end{equation}
where 
$\rho^{I}_{j}$ is the Gysin map for the inclusion map $D_{I}\hookrightarrow D_{I \backslash \{ i_{j} \}}$. 
Since the differentials $d_1$ are morphisms of pure Hodge structures, 
we have a pure Hodge structure on the subquotient 
\begin{equation}\label{eqn: weight ss E2}
E_{2}^{-m,k+m} = E_{\infty}^{-m,k+m} \simeq {\Gr}_{k+m}^{W}H^k(X). 
\end{equation}
This coincides with the Hodge structure on ${\Gr}_{k+m}^{W}H^k(X)$ induced by the Hodge filtration $F^{\bullet}$. 
%
Therefore ${\Gr}_{k+m}^{W}H^k(X)$ is a Tate twist of a subquotient of $H^{k-m}(D(m))$. 
This shows that the Hodge number 
$h_{k}^{p,q}(X)=h^{p,q}(H^k(X))$ of $H^k(X)$ 
can be nonzero only in the triangle 
\begin{equation}\label{eqn: (p,q) range}
p+q\geq k, \quad p\leq \min(k, n), \quad q\leq \min(k, n)  
\end{equation}
(see \cite{PS} Theorem 5.39). 
In particular, we have ${\Gr}^{W}_{k+m}H^k(X)\ne 0$ only when 
$k\leq k+m \leq \min(2k, 2n)$.

\subsection{Weight filtration on $F^nH^n(X)$}\label{ssec: WFkHk} 

In \S \ref{sec: weight Siegel} and \S \ref{sec: GrW}, we will be interested in the weight filtration on $F^nH^n(X)$. 
In what follows, we let $k=n$ and collect some well-known descriptions for later use. 
Although they are valid for general $k\leq n$, 
we specialize to the case $k=n$ for the convenience of reference. 
In this case, 
$\Omega^{n}_{\bar{X}}(\log D)=K_{\bar{X}}(D)$ is the logarithmic canonical bundle.  
The residue map ${\res}_{I}$ gives an isomorphism  
\begin{equation}\label{eqn: log adjunction general}
K_{\bar{X}}(D)|_{D_{I}} \simeq K_{D_{I}}(D_{I}').  
\end{equation}
This is nothing else but the well-known adjunction isomorphism. 

By the $E_1$-degeneration of the Hodge spectral sequence, 
we have a natural isomorphism 
\begin{equation}\label{eqn: FnHn K+D}
F^nH^n(X)\simeq H^0(\bar{X}, K_{\bar{X}}(D)). 
\end{equation}
The weight filtration 
$W_{\bullet}F^nH^n(X)=W_{\bullet}H^n(X)\cap F^nH^n(X)$ 
on $F^nH^n(X)$ coincides with 
the filtration on $H^0(K_{\bar{X}}(D))$ induced by the weight filtration 
$W_{\bullet} K_{\bar{X}}(D)$ on the sheaf $K_{\bar{X}}(D)$: 
\begin{equation*}\label{eqn: weight on F^kH^k}
W_{n+m}F^nH^n(X) = 
H^0(\bar{X},  W_{m}K_{\bar{X}}(D)), \qquad 0\leq m \leq n. 
\end{equation*}
(This can be seen, e.g., from \cite{PS} Theorem 3.12 (3).) 
The space $H^0(\bar{X},  W_{m}K_{\bar{X}}(D))$ 
consists of logarithmic canonical forms on $\bar{X}$ 
whose residue at every boundary stratum $D_{I}$ of codimension $|I|>m$ vanishes. 
Therefore its next subspace $H^0(\bar{X},  W_{m-1}K_{\bar{X}}(D))$ 
can be characterized inside $H^0(\bar{X},  W_{m}K_{\bar{X}}(D))$ 
as the kernel of the residue map to $D(m)$: 
\begin{equation}\label{eqn: WmFk residue kernel}
{W}_{n+m-1}F^nH^n(X) = 
\ker ( \, {W}_{n+m}F^nH^n(X) \stackrel{{\res}_{m}}{\to} H^0(D(m), K_{D(m)}) \, ). 
\end{equation}
As for the graded quotient ${\Gr}^{W}_{n+m}F^nH^n(X)$, 
we have a natural isomorphism 
\begin{equation*}
{\Gr}^{W}_{n+m}F^nH^n(X) \simeq  F^n{\Gr}^{W}_{n+m}H^n(X). 
\end{equation*}
By the weight spectral sequence and the strictness of morphisms of Hodge structures, 
we see that $F^n{\Gr}^{W}_{n+m}H^n(X)$ is the cohomology of the complex 
\begin{equation*}
F^{n-m-1}H^{n-m-2}(D(m+1))  \to  F^{n-m}H^{n-m}(D(m)) \to  F^{n-m+1}H^{n-m+2}(D(m-1)). 
\end{equation*}
(We omit to write the Tate twists.) 
The first term vanishes because the Hodge level $n-m-1$ is larger than the cohomology degree $n-m-2$. 
Therefore we have a natural isomorphism 
\begin{equation}\label{eqn: GrWF kernel}
{\Gr}^{W}_{n+m}F^nH^n(X) \: \simeq \: 
\ker ( \, H^0(D(m), K_{D(m)}) \to H^{n-m+1,1}(D(m-1)) \, ). 
\end{equation}
The composition 
\begin{equation*}
W_{n+m}F^nH^n(X) \twoheadrightarrow {\Gr}^{W}_{n+m}F^nH^n(X) 
\hookrightarrow H^0(D(m), K_{D(m)}) 
\end{equation*}
is nothing but the residue map ${\res}_{m}$ 
(as this is a process of going back from the $E_{2}$ to $E_{1}$ page, 
which amounts to sending sections of $W_{m}K_{\bar{X}}(D)$ 
to the quotient sheaf ${\rm Gr}^{W}_{m}K_{\bar{X}}(D)$).

\section{Toroidal compactifications}\label{sec: toroidal}

In this section we recall the theory of toroidal compactifications of locally symmetric varieties following \cite{AMRT}.

\subsection{Baily-Borel compactification}\label{ssec: BB}

Let ${\D}$ be a Hermitian symmetric domain and 
$\mathbb{G}$ be a connected semi-simple linear algebraic group over ${\Q}$ such that 
the identity component $G=\mathbb{G}({\R})^{\circ}$ of its real points 
is the identity component of ${\aut}({\D})$. 
Let ${\G}$ be an arithmetic subgroup of $G$. 
We assume throughout \S \ref{sec: toroidal} -- \S \ref{sec: GrW} that ${\G}$ is neat. 
(See \S \ref{sec: non-neat} for the non-neat case.) 
Let 
$X={\D}/{\G}$ 
be the locally symmetric variety defined by ${\G}$. 
A rational boundary component of ${\D}$ is called a \textit{cusp} of ${\D}$. 
Let ${\D}^{\ast}$ be the union of ${\D}$ and its cusps, equipped with the Satake topology. 
For two cusps $F\ne F'$, we write $F'\prec F$ if $F'$ is contained in the closure of $F$. 
We write $F'\preceq F$ when we want to leave the possibility $F=F'$. 

By Baily-Borel \cite{BB}, the quotient 
\begin{equation*}
X^{bb}={\D}^{\ast}/{\G} 
\end{equation*}
has the structure of a normal projective variety which contains $X$ as a Zariski open set. 
We assume throughout this paper that 
the boundary of $X^{bb}$ has codimension $\geq 2$,  
or equivalently, 
$\mathbb{G}$ has no simple ${\Q}$-factor isomorphic to ${\rm SL}_{2}/{\Q}$ 
(\cite{BB} Proposition 3.15). 
For a cusp $F$ of ${\D}$, we denote by $X_{F}$ the image of $F$ in $X^{bb}$, 
which is a locally closed subset of $X^{bb}$. 
Sometimes we use the terminology ``cusp'' also for $X_{F}$. 
If we denote by $\overline{X}_{F}$ the closure of $X_{F}$ in $X^{bb}$, 
we have a natural morphism $X_{F}^{bb}\to \overline{X}_{F}$ from 
the Baily-Borel compactification $X_{F}^{bb}$ of $X_{F}$ itself. 
This gives the normalization of $\overline{X}_{F}$. 

Let $K_{X}$ be the canonical bundle of $X$. 
This is the descent of the canonical bundle $K_{{\D}}$ of ${\D}$, 
which is a $G$-equivariant line bundle in a natural way. 
In the literature, $K_{X}$ or $K_{{\D}}$ is sometimes referred to as the automorphic line bundle of canonical weight. 
By Baily-Borel \cite{BB} (see also \cite{Mu} Proposition 3.4 (b)), 
$K_{X}$ extends to an ample line bundle on $X^{bb}$, 
which we will denote by ${\LL}$. 
Since $X^{bb}$ is normal and the boundary of $X^{bb}$ is assumed to be of codimension $\geq 2$, 
this extension is unique.

\subsection{Siegel domain realization}\label{ssec: Siegel domain}

Let $F$ be a cusp of ${\D}$. 
We denote by $N(F)$ the stabilizer of $F$ in $G$, 
$W(F)$ the unipotent radical of $N(F)$, 
$U(F)$ the center of $W(F)$, and 
$V(F)=W(F)/U(F)$. 
Then $U(F)$ and $V(F)$ are ${\R}$-linear spaces, 
and $U(F)$ is normal in $N(F)$. 
Hence $N(F)$ acts on $U(F)$ by conjugation. 
We use the notations 
\begin{equation*}
{\GFZ}=N(F)\cap {\G}, \quad 
{\UFZ}=U(F)\cap {\G}, \end{equation*}
\begin{equation*}
{\GFZdash}={\ker}({\GFZ}\to {\aut}(U(F))), \quad 
{\GF} = {\rm Im}({\GFZ}\to {\aut}(U(F))),  
\end{equation*}
\begin{equation*}
{\GFZbar} = {\GFZ}/{\UFZ}, \quad 
{\GFZbd} = {\GFZdash}/{\UFZ}. 
\end{equation*}

The linear space $U(F)$ contains an open homogeneous cone $C(F)$ preserved by $N(F)$. 
If $F'$ is another cusp with $F'\succ F$, 
then $U(F')\subset U(F)$ and $C(F')$ is a rational boundary component of $C(F)$. 
We denote by $C(F)^{\ast}\subset U(F)$ the union of $C(F)$ and 
all such rational boundary components (including $\{ 0 \}$ for which $F'={\D}$ by convention). 

According to \cite{AMRT} \S III.4.3, 
there is a complex manifold ${\D}(F)$ containing ${\D}$ as an analytic open set 
and acted on by $U(F)_{{\C}}\cdot N(F)$, 
and a $U(F)_{{\C}}\cdot N(F)$-equivariant two-step fibration 
\begin{equation}\label{eqn: Siegel domain}
{\D}(F) \to {\VF} \to F, 
\end{equation}
such that 
${\D}(F) \to {\VF}$ is a principal $U(F)_{{\C}}$-bundle and 
${\VF} \to F$ is an affine space bundle whose fibers are acted on by $V(F)$ freely and transitively. 
Moreover, there is a $U(F)_{{\C}}\cdot N(F)$-equivariant map ${\D}(F)\to U(F)$ 
such that ${\D}\subset {\D}(F)$ is the inverse image of $C(F)$. 
Thus the fiber of ${\D}\subset {\D}(F)$ over each point of ${\VF}$ is isomorphic to the tube domain 
$U(F)+iC(F)\subset U(F)_{{\C}}$ up to a choice of a base point. 
The two-step fibration \eqref{eqn: Siegel domain} is the \textit{Siegel domain realization} of ${\D}$ 
(of the third kind) with respect to $F$. 

We have $X_{F} \simeq F/{\GFZ}$. 
The group ${\GFZbd}$ acts on ${\VF}$ freely, 
and the quotient 
\begin{equation*}
Y_{F} = {\VF}/{\GFZbd} 
\end{equation*}
is an abelian fibration over the finite cover $F/{\GFZdash}$ of $X_{F}$ 
(cf.~\cite{AMRT} p.174).

\subsection{Toroidal compactifications}\label{ssec: toroidal}

A toroidal compactification of $X$ can be constructed by choosing 
a ${\G}$-admissible collection of fans $\Sigma=(\Sigma_{F})_{F}$ in the sense of 
\cite{AMRT} Definition III.5.1. 
Each fan $\Sigma_{F}$ is a ${\G}_{F}$-admissible rational polyhedral cone decomposition of $C(F)^{\ast}$, 
with the ${\Z}$-structure of $U(F)$ given by ${\UFZ}$. 
(In our convention, polyhedral cones are closed.) 
Let 
\begin{equation*}
{\XF} = {\D}/{\UFZ}, \quad \mathcal{T}(F)={\D}(F)/{\UFZ}, \quad T(F)=U(F)_{{\C}}/{\UFZ}. 
\end{equation*}
Then $T(F)$ is an algebraic torus, $\mathcal{T}(F)$ is a principal $T(F)$-bundle over ${\VF}$, 
and ${\XF}$ is an analytic open set of $\mathcal{T}(F)$. 
Let $T(F)^{\Sigma_{F}}$ be the torus embedding of $T(F)$ defined by the fan $\Sigma_{F}$ and 
let $\mathcal{T}(F)^{\Sigma_{F}}$ be the fiber bundle $\mathcal{T}(F)\times_{T(F)} T(F)^{\Sigma_{F}}$. 
Then let ${\XFcpt}$ be the interior of the closure of ${\XF}$ in $\mathcal{T}(F)^{\Sigma_{F}}$. 
This is the partial toroidal compactification in the direction of $F$. 
At each fiber over ${\VF}$, we are taking the interior of the closure of 
$(U(F)+iC(F))/{\UFZ}\subset T(F)$ in $T(F)^{\Sigma_{F}}$. 

The torus embedding $T(F)^{\Sigma_{F}}$ is stratified into $T(F)$-orbits (cf.~\cite{AMRT} \S I.1): 
\begin{equation*}
T(F)^{\Sigma_{F}} = \bigsqcup_{\sigma\in \Sigma_{F}} \mathbb{O}(\sigma). 
\end{equation*}
Each stratum $\mathbb{O}(\sigma)$ corresponds to a cone $\sigma$ in $\Sigma_{F}$ 
and is isomorphic to the quotient torus of $T(F)$ by the sub torus defined by the span of $\sigma$. 
Accordingly, $\mathcal{T}(F)^{\Sigma_{F}}$ has a stratification 
\begin{equation*}
\mathcal{T}(F)^{\Sigma_{F}} = \bigsqcup_{\sigma\in \Sigma_{F}} \mathcal{T}(F)^{\sigma}, 
\end{equation*}
where each stratum $\mathcal{T}(F)^{\sigma}$ is a principal $\mathbb{O}(\sigma)$-bundle over ${\VF}$. 
Taking the intersection with ${\XFcpt}$, 
this defines a stratification of ${\XFcpt}$: 
\begin{equation*}
{\XFcpt} = \bigsqcup_{\sigma\in \Sigma_{F}} {\SFs} 
\end{equation*}
where 
${\SFs}= {\XFcpt}\cap \mathcal{T}(F)^{\sigma}$. 

We say that a cone $\sigma\in \Sigma_{F}$ is an \textit{$F$-cone} 
if the interior of $\sigma$ is contained in $C(F)$, or equivalently, 
$\sigma \cap C(F)\ne \emptyset$. 
In this case, 
$\mathbb{O}(\sigma)$ is contained in the interior of the closure of 
$(U(F)+iC(F))/{\UFZ}$ in $T(F)^{\Sigma_{F}}$ 
by the description of $\mathbb{O}(\sigma)$ in \cite{AMRT} \S I.1, 
so we have $\mathcal{T}(F)^{\sigma}\subset {\XFcpt}$. 
Thus, if $\sigma$ is an $F$-cone, 
${\SFs}=\mathcal{T}(F)^{\sigma}$ is a principal $\mathbb{O}(\sigma)$-bundle over ${\VF}$. 
We denote by $\Sigma_{F}^{\circ}\subset \Sigma_{F}$ the subset of $F$-cones. 

If $F'$ is another cusp with $F'\succ F$, 
the inclusion $U(F')_{{\Z}}\subset {\UFZ}$ defines a free quotient map $\mathcal{X}(F')\to {\XF}$ 
by ${\UFZ}/U(F')_{{\Z}}$. 
This extends to an \'etale map $\mathcal{X}(F')^{\Sigma_{F'}}\to {\XFcpt}$. 
This maps a boundary stratum 
$S\!(F', \sigma')$ of $\mathcal{X}(F')^{\Sigma_{F'}}$ 
onto the boundary stratum 
$S\!(F, \sigma')$ of ${\XFcpt}$ 
where we view $\sigma'\in \Sigma_{F'}$ as a (boundary) cone of $\Sigma_{F}$ 
by the inclusion $\Sigma_{F'}\subset \Sigma_{F}$ (cf.~\cite{AMRT} Proof of Lemma III.5.5). 

Now the toroidal compactification of $X$ is defined by 
\begin{equation*}
{\Xcpt} = \left( {\D}\sqcup \bigsqcup_{F}{\XFcpt} \right) / \sim 
\end{equation*}
where $F$ ranges over all cusps of ${\D}$ 
and $\sim$ is the equivalence relation generated by the following \'etale maps: 
\begin{itemize}
\item The ${\G}$-action ${\D}\to {\D}$ and ${\XFcpt}\to \mathcal{X}(\gamma F)^{\Sigma_{\gamma F}}$.  
\item The gluing maps ${\D}\to {\XFcpt}$ and $\mathcal{X}(F')^{\Sigma_{F'}}\to {\XFcpt}$ for $F'\succ F$. 
\end{itemize}
By \cite{AMRT} Theorem III.5.2, ${\Xcpt}$ is a compact Moishezon space 
containing $X$ as a Zariski open set, 
and we have a morphism 
$\pi\colon {\Xcpt}\to X^{bb}$ to the Baily-Borel compactification 
which extends the identity of $X$. 

We may choose $\Sigma$ so that ${\Xcpt}$ is smooth and projective 
(\cite{AMRT} Corollary IV.2.4). 
Here the smoothness is assured by requiring that 
every cone $\sigma\in \Sigma_{F}$ is generated by a part of a basis of ${\UFZ}$, 
and the projectivity is assured by requiring the existence of the so-called polarization functions 
(\cite{AMRT} Definition IV.2.1). 
We call such $\Sigma$ smooth and projective respectively. 
In this case, the boundary $D$ of ${\Xcpt}$ is a divisor with normal crossings. 
Then, if ${\LL}$ is the automorphic line bundle of canonical weight extended over $X^{bb}$, 
Mumford proved that 
\begin{equation}\label{eqn: L = K+D}
\pi^{\ast}{\LL} \simeq K_{{\Xcpt}}(D) 
\end{equation}
(\cite{Mu} Proposition 3.4). 
We may choose $\Sigma$ so that $D$ is furthermore simple normal crossing 
(see Appendix \ref{sec: SNC}). 
In the rest of this paper, we will assume that $\Sigma$ is chosen so.

\section{Vanishing of Hodge numbers I}\label{sec: VT I}

Let $X={\D}/{\G}$ be a locally symmetric variety as in \S \ref{sec: toroidal} with ${\G}$ neat. 
We write $n=\dim X$. 
In this section we study the mixed Hodge structure on $H^k(X)$ for general $k$.  
As recalled in \S \ref{ssec: weight ss}, ${\Gr}^{W}_{k+m}H^{k}(X)$ can be nonzero 
only in the range $k\leq k+m \leq \min(2k, 2n)$. 
The first subspace ${\Gr}^{W}_{k}H^{k}(X)=W_{k}H^{k}(X)$ is the image of the natural map 
$H^k_{(2)}(X)\to H^k(X)$ from the $L^2$-cohomology of $X$ (\cite{HZIII} Remark 5.5). 
While this part usually encodes rich information, 
the higher graded pieces ${\Gr}^{W}_{k+m}H^{k}(X)$, $m>0$, also reflect the geometry of $X$. 
Let us begin by recalling the following well-known property. 

\begin{lemma}[cf.~\cite{HZIII}]\label{lem: pure BB}
Let $c$ be the codimension of the boundary of the Baily-Borel compactification $X^{bb}$. 
If $k<c$, we have $H^{k}(X)=W_{k}H^{k}(X)$, i.e., 
the mixed Hodge structure on $H^k(X)$ is pure of weight $k$. 
\end{lemma}

\begin{proof}
This is essentially noticed in \cite{HZIII} Remark 5.5. 
For the convenience of the reader, we give a direct argument here. 
Let $IH^k(X^{bb})$ be the intersection cohomology of $X^{bb}$. 
The restriction map 
$IH^k(X^{bb})\to IH^k(X)=H^k(X)$ 
is an isomorphism if $k<c$ (\cite{Max} Theorem 6.7.4 (a)). 
By the theory of mixed Hodge modules (\cite{Sa}), 
this is a morphism of mixed Hodge structures, 
and the mixed Hodge structure on $IH^{k}(X^{bb})$ 
is pure of weight $k$ by the projectivity of $X^{bb}$. 
\end{proof}

From now on, we consider the case $k\geq c$. 
We define a sequence of natural numbers 
\begin{equation}\label{eqn: n(i)}
0 < n(1) <n(2) < \cdots < n(r) \leq n 
\end{equation}
as the set of values of $\dim U(F)$ for all cusps $F$ of ${\D}$. 
The equality $n(r)=n$ holds if and only if there is a cusp $F$ with $\dim {\VF} = 0$, 
in which case the Siegel domain realization at $F$ 
realizes ${\D}$ as a tube domain in $U(F)_{{\C}}$. 
We also set $n(0)=0$. 

When the algebraic group $\mathbb{G}$ is ${\Q}$-simple, 
$r$ is equal to the ${\Q}$-rank of $\mathbb{G}$. 
We have a flag $F_1\succ F_2 \succ \cdots \succ F_r$ of reference cusps of ${\D}$ of length $r$ such that 
every cusp of ${\D}$ is equivalent to one of $F_{i}$ 
under the action of $\mathbb{G}({\Q})$ (\cite{BB} Theorem 3.8). 
Then $n(i)=\dim U(F_{i})$, and we call the index $i$ the \textit{corank} of $F_{i}$. 
On the other hand, when $\mathbb{G}$ is not ${\Q}$-simple, 
cusps of ${\D}$ are product of cusps of some factors of ${\D}$ with the remaining factors. 
So we have a tree of reference cusps, rather than a single flag. 
Then $r$ is no longer equal to the ${\Q}$-rank of $\mathbb{G}$ in general, 
and there can be two non-equivalent cusps $F$, $F'$ with $\dim U(F)=\dim U(F')$. 
In this case, the terminology ``corank'' would be no longer appropriate. 
Nevertheless \eqref{eqn: n(i)} is useful for studying the mixed Hodge structures of $X$. 

Recall from \eqref{eqn: (p,q) range} that the Hodge number $h^{p,q}_{k}(X)$ for $p+q\geq k$ 
can be nonzero only when $|p-q| \leq \min (2k, 2n)-p-q$. 
We shall prove the following further vanishing. 
We say that a pure Hodge structure $V$ is \textit{effective} 
if in the Hodge decomposition $V=\bigoplus_{p,q}V^{p,q}$ 
only indices $(p, q)$ with $p\geq 0$, $q\geq 0$ appear. 

\begin{theorem}\label{thm: VT I}
Let $1\leq i \leq r$. 
If $n(i-1) < m \leq n(i)$, 
then ${\Gr}^W_{k+m}H^k(X)$ is a Tate twist of an effective pure Hodge structure of weight $\leq n-n(i)$. 
Therefore, if $n(i-1)< p+q-k \leq n(i)$, we have 
$h_{k}^{p,q}(X)\ne 0$ only when 
\begin{equation}\label{eqn: Hodge weight min}
|p-q| \: \leq \: \min (n-n(i), \, 2k-p-q, \, 2n-p-q). 
\end{equation}
When $m>n(r)$, we have ${\Gr}^W_{k+m}H^k(X)=0$. 
\end{theorem}

In Figure \ref{figure: Hodge terrace}, 
we have drawn a shape of the range of such $(p, q)$ in the case $k=n$. 
In this case, the minimum in \eqref{eqn: Hodge weight min} is always attained by $n-n(i)$. 
When $k<n$, the range of $(p, q)$ is the intersection of 
the $(k-n)$-shift of the stairs in Figure \ref{figure: Hodge terrace} with the smaller triangle $p+q\geq k$, $p, q\leq k$. 
We draw this intersection in Figure \ref{figure: Hodge terrace k<n}. 
In Proposition \ref{thm: VT II}, we will show moreover that the roof of the smaller triangle vanishes.  
Similarly, when $k>n$, the range of $(p, q)$ is the intersection of 
the $(k-n)$-shift of the stairs in Figure \ref{figure: Hodge terrace} with the triangle $p+q\geq k$, $p, q\leq n$. 

\begin{figure}[h]
\includegraphics[height=40mm, width=80mm]{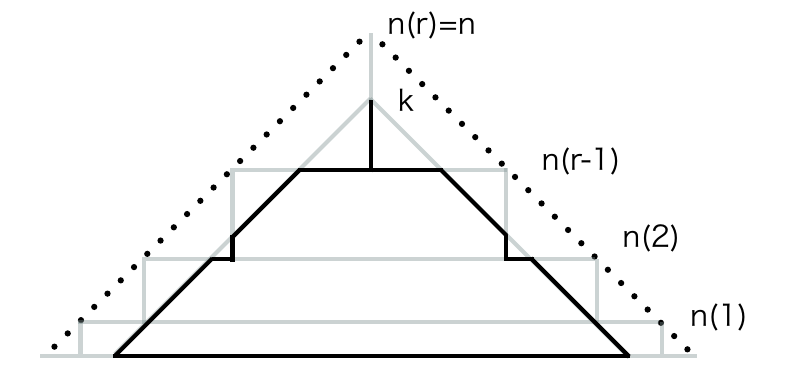}
\caption{Range of $(p, q)$ when $k<n$}\label{figure: Hodge terrace k<n}
\end{figure}

The rest of this section is devoted to the proof of Theorem \ref{thm: VT I}. 
\S \ref{ssec: stratification} and \S \ref{ssec: lemma MHS} are preliminaries, 
and the proof is completed in \S \ref{ssec: proof VT I}.

\subsection{Stratification of toroidal compactification}\label{ssec: stratification}

Let ${\Xcpt}$ be a smooth projective toroidal compactification of $X={\D}/{\G}$ 
with simple normal crossing (SNC) boundary divisor $D={\Xcpt}-X$. 
In this subsection we explain the structure of a natural stratification of $D$ following \cite{AMRT} \S III.5 -- \S III.7,  
and prove Lemma \ref{lem: D(m) toroidal} and Lemma \ref{lem: boundary DFs}. 
These two lemmas are the basis for applying mixed Hodge theory to toroidal compactifications. 

Let $F$ be a cusp of ${\D}$. 
By construction we have natural maps 
\begin{equation*}
{\XFcpt} \to {\XFcpt}/{\GFZbar} \to {\Xcpt}. 
\end{equation*}
The second map is an isomorphism in a neighborhood of the $F$-locus 
$\pi^{-1}(X_{F})$ in ${\Xcpt}$ (see \cite{AMRT} p.175, p.184). 
If $\sigma$ is an $F$-cone, the boundary stratum ${\SFs}$ of ${\XFcpt}$ is mapped to a sublocus of $\pi^{-1}(X_{F})$  
(see \cite{AMRT} Proof of Lemma III.5.5 and p.171). 
On the other hand, if $\sigma\in \Sigma_{F}$ is not an $F$-cone, ${\SFs}$ is mapped to 
a sublocus of $\pi^{-1}(X_{F'})$ where $F'\succ F$ is the cusp with $\sigma\in \Sigma_{F'}^{\circ} \subset \Sigma_{F}$. 
Therefore we have an isomorphism 
\begin{equation*}
\left( \bigsqcup_{\sigma \in \Sigma_{F}^{\circ}} {\SFs} \right) / {\GFZbar} \simeq \pi^{-1}(X_{F}). 
\end{equation*}
For an $F$-cone $\sigma\in \Sigma_{F}^{\circ}$ 
we denote by ${\DFs}$ the image of ${\SFs}$ in ${\Xcpt}$.  
This is a locally closed subset of ${\Xcpt}$ of codimension $\dim \sigma$. 
We thus have the stratification 
\begin{equation}\label{eqn: stratification toroidal}
{\Xcpt} = X \sqcup \bigsqcup_{(F,\sigma)}{\DFs}, 
\end{equation}
where $(F, \sigma)$ ranges over ${\G}$-equivalence classes of pairs of 
a cusp $F$ and an $F$-cone $\sigma$. 
The closure of each stratum ${\DFs}$ is stratified as 
\begin{equation}\label{eqn: closure stratum}
{\DFsbar} = {\DFs} \sqcup \bigsqcup_{(F', \sigma')}D(F', \sigma'), 
\end{equation}
where $(F', \sigma')$ ranges over pairs such that 
$F' \preceq F$ and $\sigma'\in \Sigma_{F'}^{\circ}$ with $\sigma\prec \sigma'$ 
via the inclusion $U(F)\subseteq U(F')$. 

Each stratum ${\DFs}$ has the following structure. 
Since ${\GF}$ is neat, the stabilizer of the $F$-cone $\sigma$ in ${\GF}$ is trivial. 
Therefore we have 
\begin{equation*}
{\DFs} \simeq {\SFs}/{\GFZbd} 
\end{equation*}
(cf.~\cite{AMRT} p.184). 
Since ${\SFs}$ is a principal $\mathbb{O}(\sigma)$-bundle over ${\VF}$ and ${\GFZbd}$ acts on ${\VF}$ freely, 
${\DFs}$ is a principal $\mathbb{O}(\sigma)$-bundle over $Y_{F}$  
(in the \'etale topology $\Leftrightarrow$ in the Zariski topology).  
Note that the projection ${\DFs}\to Y_F$ depends on the cone $\sigma$, rather than its ${\GF}$-equivalence class. 
This issue does not matter for the moment, but will be taken up in \S \ref{ssec: GF monodromy}.

We connect this toroidal compactification with the setting of mixed Hodge theory in \S \ref{sec: MHS}. 
By \eqref{eqn: stratification toroidal}, 
irreducible components of $D={\Xcpt}-X$ are of the form $\overline{D(F, \tau)}$ with $\dim \tau=1$. 
We call a $1$-dimensional cone a \textit{ray}. 
Thus irreducible components of $D$ correspond to 
${\G}$-equivalence classes of rays in $\Sigma$, 
where if $F\prec F'$ we identify $\Sigma_{F'} \subset \Sigma_{F}$ naturally.

\begin{lemma}\label{lem: D(m) toroidal}
Let $F$ be a cusp of ${\D}$ and $\sigma\in \Sigma_{F}^{\circ}$ be an $F$-cone of dimension $m$. 
Let $\tau_1, \cdots, \tau_{m}\in \Sigma_{F}$ be the rays of $\sigma$. 
Then $\overline{{\DFs}}$ is smooth and is a connected component of   
$\overline{D(F_1, \tau_1)} \cap \cdots \cap \overline{D(F_m, \tau_m)}$ 
where $F_i\succeq F$ is the cusp with $\tau_{i}\in \Sigma_{F_{i}}^{\circ}$. 
The rays $\tau_1, \cdots, \tau_{m}$ are not ${\G}$-equivalent to each other. 

Conversely, every connected component of a non-empty intersection of 
some irreducible components of $D$ is of the form ${\DFsbar}$ for 
some cusp $F$ and an $F$-cone $\sigma\in \Sigma_{F}^{\circ}$. 
\end{lemma}

\begin{proof}
We begin with the first part. 
In ${\XFcpt}$, the divisor $\overline{S\!(F, \tau_1)} + \cdots + \overline{S\!(F, \tau_m)}$ 
is simple normal crossing and we have 
\begin{equation}\label{eqn: SFs intersection}
\overline{S\!(F, \sigma)} = \overline{S\!(F, \tau_1)} \cap \cdots \cap \overline{S\!(F, \tau_m)}. 
\end{equation}
Sending this equality by ${\XFcpt}\to {\Xcpt}$ and taking the closures in ${\Xcpt}$, 
we see that 
 \begin{equation}\label{eqn: DFs intersection}
{\DFsbar} \subset \overline{D(F_1, \tau_1)} \cap \cdots \cap \overline{D(F_m, \tau_m)} 
\end{equation}
in ${\Xcpt}$. 
In what follows, we write $D_{i}=\overline{D(F_i, \tau_i)}$. 
We first check that the divisors $D_{1}, \cdots, D_{m}$ are distinct. 
Since the projection ${\XFcpt}\to {\Xcpt}$ is a local isomorphism in a neighborhood of every point of ${\SFs}$, 
the intersection of $D_{1}+ \cdots +D_{m}$ with a neighborhood of a point of ${\DFs}$ has $m$ distinct branches. 
Since each $D_{i}$ does not have self-intersection by our SNC condition, 
this shows that $D_{i}\ne D_{j}$ if $i\ne j$. 
Thus $\tau_1, \cdots, \tau_{m}$ are not ${\G}$-equivalent to each other. 

Now, in \eqref{eqn: DFs intersection}, ${\DFsbar}$ is irreducible of codimension $m$, 
while $D_{1} \cap \cdots \cap D_{m}$ is smooth and equidimensional of codimension $m$ by our SNC condition. 
This implies that ${\DFsbar}$ is a connected component of $D_{1} \cap \cdots \cap D_{m}$, 
and in particular smooth. 

Next we prove the second part. 
We take $m$ divisors $D_{i}=\overline{D(F_i, \tau_i)}$, $1\leq i \leq m$, 
such that $D_{1} \cap \cdots \cap D_{m}$ is non-empty. 
Since each $D_i$ is stratified like \eqref{eqn: closure stratum}, 
their intersection is also a union of some strata: 
\begin{equation*}
D_{1} \cap \cdots \cap D_{m} = \bigsqcup_{\alpha} D(F_{\alpha}, \sigma_{\alpha}). 
\end{equation*}
Among the strata in the right hand side, we pick up those of codimension $m$. 
Then $D_{1} \cap \cdots \cap D_{m}$ coincides with 
the disjoint union of the closures of these strata, 
because every connected component of $D_{1} \cap \cdots \cap D_{m}$ is irreducible of codimension $m$ 
by our SNC condition. 
\end{proof}

\begin{remark}\label{rmk: non-connected}
In general, $\overline{D(F_1, \tau_1)} \cap \cdots \cap \overline{D(F_m, \tau_m)}$ 
may have several connected components. 
For example, consider a cusp of a Hilbert modular surface 
whose minimal resolution is a union of two ${\proj}^1$'s meeting at two points. 
Equality can be lost when passing from \eqref{eqn: SFs intersection} to \eqref{eqn: DFs intersection}. 
\end{remark} 

By Lemma \ref{lem: D(m) toroidal}, 
the space $D(m)$ considered in \eqref{eqn: D(m)} 
is described in the present setting as 
\begin{equation}\label{eqn: D(m) toroidal}
D(m) = \bigsqcup_{\substack{(F,\sigma) \\ \dim \sigma=m}} {\DFsbar}, 
\end{equation}
where $(F, \sigma)$ ranges over ${\G}$-equivalence classes of pairs of 
a cusp $F$ and an $F$-cone $\sigma$ with $\dim \sigma=m$. 
Indices $I$ in \eqref{eqn: D(m)} correspond to 
tuples $\{ [\tau_1], \cdots, [\tau_{m}] \}$ 
of $m$ distinct ${\G}$-equivalence classes of rays in $\Sigma$. 
By Lemma \ref{lem: D(m) toroidal}, we have $D_{I}\ne \emptyset$ if and only if 
there exists an $m$-dimensional cone $\sigma\in \Sigma$ 
whose rays belong to $[\tau_1], \cdots, [\tau_{m}]$. 
Via the assignment 
\begin{equation*}\label{eqn: simplicial map}
\sigma\mapsto I =  \{ [\tau_1], \cdots, [\tau_{m}] \}, 
\end{equation*}
the simplex structure of $\sigma$ agrees with that of $I$. 
Lemma \ref{lem: D(m) toroidal} amounts to saying that 
the $\Delta$-complex (in the sense of \cite{Hatcher} \S 2.1) 
associated to the projectivization of $\Sigma/{\G}$ is exactly the dual complex of the SNC divisor $D$. 

The following property will be used in \S \ref{ssec: weight=corank} (but not in this \S \ref{sec: VT I}). 

\begin{lemma}\label{lem: boundary DFs}
Let ${\DFs}$ be as in Lemma \ref{lem: D(m) toroidal} and let ${\DFs}'={\DFsbar}-{\DFs}$. 
Then ${\DFs}'$ coincides with the intersection of ${\DFsbar}$ with 
the union of the irreducible components of $D$ that do not contain ${\DFsbar}$. 
\end{lemma}

\begin{proof}
Let $\tau_1, \cdots, \tau_m$ be the rays of $\sigma$ as in Lemma \ref{lem: D(m) toroidal}. 
By Lemma \ref{lem: D(m) toroidal} and the normal crossing property, 
any irreducible component $D_{\alpha}$ of $D$ different from 
$\overline{D(F_1, \tau_1)}, \cdots, \overline{D(F_m, \tau_m)}$ does not contain ${\DFsbar}$. 
Since $D_{\alpha}$ is a union of some strata, it must be disjoint from ${\DFs}$, 
so its intersection with ${\DFsbar}$ is contained in ${\DFs}'$. 

Conversely, let $D(F', \sigma')$ be a stratum of ${\DFs}'$ as in \eqref{eqn: closure stratum}. 
We take a ray $\tau$ of $\sigma'$ not contained in $\sigma$ and 
let $\sigma''$ be the cone spanned by $\sigma$ and $\tau$. 
Thus $\sigma\prec \sigma'' \preceq \sigma'$.  
Then we have 
\begin{equation*}
\overline{D(F', \sigma')} \, \subset \, \overline{D(F'', \sigma'')} \, \subset \, {\DFsbar}, 
\end{equation*}
where $F''$ is the cusp with $\sigma''\in \Sigma_{F''}^{\circ}$. 
By Lemma \ref{lem: D(m) toroidal} for $\sigma''$, 
we see that $\overline{D(F'', \sigma'')}$ is a connected component of 
${\DFsbar} \cap \overline{D(F_{\tau}, \tau)}$,  
where $F_{\tau}$ is the cusp with $\tau\in \Sigma_{F_{\tau}}^{\circ}$. 
Thus $D(F', \sigma')$ is contained in the proper intersection 
${\DFsbar} \cap \overline{D(F_{\tau}, \tau)}$. 
This proves our assertion. 
\end{proof}

\subsection{Weights of the strata}\label{ssec: lemma MHS}

In this subsection we prepare two lemmas 
for bounding the Hodge levels of the cohomology of the strata ${\DFsbar}$. 
Both are standard and we state them for the convenience of reference. 
In what follows, we consider the abelian category of mixed Hodge structures. 

\begin{lemma}\label{lem: weight torus bundle}
Let $V\to Y$ be a principal $({\C}^{\ast})^d$-bundle with $V, Y$ smooth algebraic varieties. 
Then $H^k(V)$ has a filtration whose graded quotients are 
Tate twists of subquotients of  $H^j(Y)$, $k-d\leq j\leq k$. 
\end{lemma}

\begin{proof}
By taking the quotient by ${\C}^{\ast}$ repeatedly, 
the assertion reduces to the case $d=1$, namely $Y=V/{\C}^{\ast}$. 
Let $\bar{V}$ be the fiber bundle $V\times_{{\C}^{\ast}} \mathbb{A}^1$. 
Then $\bar{V}\to Y$ is an $\mathbb{A}^1$-bundle with zero section $\bar{V}-V \simeq Y$. 
By the cohomology exact sequence for the pair $(\bar{V}, V)$ (see \cite{PS} Corollary 5.50), 
we have an exact sequence  
\begin{equation*}
H^k(\bar{V}) \to H^k(V) \to H^{k+1}_{Y}(\bar{V}) 
\end{equation*}
of mixed Hodge structures. 
We have the Thom isomorphism $H^{k+1}_{Y}(\bar{V}) \simeq H^{k-1}(Y)(-1)$ 
and the homotopy invariance $H^k(\bar{V})\simeq H^k(Y)$. 
This defines a filtration on $H^k(V)$ as desired. 
\end{proof}

In fact, the cohomology exact sequence for $(\bar{V}, V)$ splits, but we will not use this information. 

\begin{lemma}\label{lem: weight stratification}
Let $D$ be a smooth algebraic variety and 
$\emptyset=U_0 \subset U_1 \subset \cdots \subset U_N=D$ 
be a sequence of Zariski open subsets. 
Suppose that $V_i=U_i-U_{i-1}$ is smooth and equidimensional. 
Let $c_i$ be the codimension of $V_{i}$. 
Then $H^k(D)$ has a filtration whose graded quotients are Tate twists of subquotients of 
$H^{k-2c_{i}}(V_i)$, $1\leq i \leq N$. 
\end{lemma}

\begin{proof}
By the cohomology exact sequence for the pair $(U_i, U_{i-1})$ 
and the Thom isomorphism, we have an exact sequence 
\begin{equation*}
H^{k-2c_{i}}(V_{i})(-c_i) \to H^k(U_i) \to H^k(U_{i-1}) 
\end{equation*}
of mixed Hodge structures. 
By considering inductively on $i$, this implies the assertion for $U_N=D$. 
\end{proof}

\subsection{Proof of Theorem \ref{thm: VT I}}\label{ssec: proof VT I}

Now we prove Theorem \ref{thm: VT I}. 
As in \S \ref{ssec: stratification}, 
we take a smooth projective toroidal compactification ${\Xcpt}$ of $X={\D}/{\G}$ 
with SNC boundary divisor. 
We compute the weight spectral sequence with this compactification. 
Let $0<m\leq n$. 
If $F$ is a cusp of ${\D}$ and $\sigma$ is an $F$-cone with $\dim \sigma = m$, we have  
$m \leq \dim U(F) \leq n(r)$. 
By \eqref{eqn: D(m) toroidal}, we see that $D(m)$ is empty if $m>n(r)$. 
By the weight spectral sequence, we have ${\Gr}^{W}_{k+m}H^k(X)=0$ in this case. 
In what follows, we consider the case $m\leq n(r)$. 
Let $1\leq i \leq r$ be the index with $n(i-1)<m \leq n(i)$. 

Let $(E_{1}^{p,q})$ be the $E_1$ page of the weight spectral sequence. 
By \eqref{eqn: weight E1} and \eqref{eqn: D(m) toroidal}, we have 
\begin{equation}\label{eqn: E_1 toroidal}
E_1^{-m,k+m} = H^{k-m}(D(m))(-m) 
= \bigoplus_{\substack{(F,\sigma) \\ \dim \sigma=m}} H^{k-m}({\DFsbar})(-m). 
\end{equation}
We apply Lemma \ref{lem: weight stratification} to the stratification \eqref{eqn: closure stratum} of ${\DFsbar}$, 
where we take the Zariski open set $U_{i}\subset {\DFsbar}$ to be 
the complement of the union of strata of codimension $\geq i$. 
This shows that 
$H^{k-m}({\DFsbar})$ has a filtration whose graded quotients are 
Tate twists of subquotients of $H^{k-m-\ast}(D(F', \sigma'))$ for $F'\preceq F$, $\sigma'\succeq \sigma$. 
Since each stratum $D(F', \sigma')$ is a principal torus bundle over $Y_{F'}$, 
we can apply Lemma \ref{lem: weight torus bundle}. 
This tells us that 
$H^{k-m-\ast}(D(F', \sigma'))$ has a filtration whose graded quotients are 
Tate twists of subquotients of $H^{\bullet}(Y_{F'})$ with $\bullet\leq k-m-\ast$. 
Combining these two constructions, we see that 
$H^{k-m}({\DFsbar})$ has a filtration whose graded quotients are 
Tate twists of subquotients of $H^{\bullet}(Y_{F'})$ 
for various cusps $F'\preceq F$ and degrees $\bullet$. 

We have 
\begin{equation*}
\dim Y_{F'} \leq \dim D(F', \sigma') \leq \dim {\DFs} = n-m < n-n(i-1). 
\end{equation*}
Since $\dim Y_{F'}=n-\dim U(F')$ itself is equal to $n-n(\alpha)$ for some $1\leq \alpha \leq r$ 
by the definition \eqref{eqn: n(i)} of $n(\ast)$, 
we see that $\dim Y_{F'} \leq n-n(i)$. 
Hence every weight graded quotient of $H^{\bullet}(Y_{F'})$ 
is a Tate twist of an effective pure Hodge structure of weight $\leq n-n(i)$. 

Summing up the arguments so far, 
we see that the pure Hodge structure 
$H^{k-m}({\DFsbar})$ has a filtration whose graded quotients are 
Tate twists of effective pure Hodge structures of weight $\leq n-n(i)$. 
Therefore $H^{k-m}({\DFsbar})$ itself is a Tate twist of an effective pure Hodge structure of weight $\leq n-n(i)$. 
By \eqref{eqn: E_1 toroidal}, the same holds for $E_1^{-m,k+m}$. 
Passing to a subquotient, 
we see that the same holds for $E_2^{-m,k+m}={\Gr}_{k+m}^WH^k(X)$. 
This proves Theorem \ref{thm: VT I}.

\section{Vanishing of Hodge numbers II}\label{sec: VT II}

Let $X={\D}/{\G}$ be a locally symmetric variety as in \S \ref{sec: toroidal} with ${\G}$ neat. 
In this section we prove our second vanishing theorem and derive some consequences.  
This is of supplementary nature and 
independent of \S \ref{sec: VT I}, \S \ref{sec: weight Siegel} and \S \ref{sec: GrW}. 

Let $k<n$. 
The weight filtration 
$W_{\bullet}F^kH^k(X) = W_{\bullet}H^k(X)\cap F^kH^k(X)$ 
on $F^kH^k(X)$ has level a priori from $k$ to $2k$. 
In fact, it is trivial: 

\begin{proposition}\label{thm: VT II}
Let $k<n$. 
Then $F^kH^k(X)=W_kF^kH^k(X)$. 
Therefore we have $h_{k}^{k,m}(X)=h_{k}^{m,k}(X)=0$ for $m>0$. 
\end{proposition}

\begin{proof}
We take a smooth projective toroidal compactification ${\Xcpt}$ of $X$ with SNC boundary divisor $D$. 
By the assumption that the Baily-Borel boundary has codimension $c>1$, the Koecher principle holds (\cite{Ba}). 
Then the Pommerening extension theorem (\cite{Po}) holds which says that 
every holomorphic $k$-form on $X$ extends holomorphically over ${\Xcpt}$. 
This means that the inclusions 
\begin{equation*}
H^0({\Xcpt}, \Omega^k) \subset H^0({\Xcpt}, \Omega^k(\log D)) \subset H^0(X, \Omega^k) 
\end{equation*}
are all equalities. 
Since $F^kH^k(X)$ is the image of 
$H^0({\Xcpt}, \Omega^k(\log D))\to H^k(X)$ 
and $W_kF^kH^k(X)$ is the image of 
$H^0({\Xcpt}, \Omega^k)\to H^k(X)$, 
we obtain $F^kH^k(X)=W_kF^kH^k(X)$.  
\end{proof}

This in particular implies the following. 

\begin{corollary}\label{cor: VT II}
If $2\leq k <n$, then ${\Gr}_{2k}^{W}H^k(X)= {\Gr}_{2k-1}^{W}H^k(X)=0$. 
In particular, $H^2(X)$ is pure. 
\end{corollary}

When $c>2$, the purity of $H^2(X)$ also follows from Lemma \ref{lem: pure BB}. 
In the case of Siegel modular $3$-folds, the purity of $H^2(X)$ was first proved by Oda-Schwermer \cite{OS1}. 
Our proof of Proposition \ref{thm: VT II} is a direct generalization of their argument. 
Similarly, $H^1(X)$ is pure either by Lemma \ref{lem: pure BB} or Proposition \ref{thm: VT II}, 
but in fact $H^1(X)$ vanishes in many cases (\cite{Mar}). 

The purity of $H^1(X)$ and $H^2(X)$ in turn has the following consequence. 

\begin{proposition}
Let ${\Xcpt}$ be a smooth projective toroidal compactification of $X$ with SNC boundary divisor $D$. 
Let $D=D_{1}+\cdots +D_{\alpha}$ be the irreducible decomposition of $D$. 

(1) $D_{1}, \cdots, D_{\alpha}$ are linearly independent in $H^2({\Xcpt})$, and we have  
\begin{equation*}
H^2({\Xcpt}) / {\C}D_{1} \oplus \cdots \oplus {\C}D_{\alpha} \: \simeq \: H^2(X). 
\end{equation*}

(2) The Gysin map 
$\bigoplus_{i=1}^{\alpha} H^1(D_{i}) \to H^3({\Xcpt})$ 
is injective. 
\end{proposition}

\begin{proof}
(1) Let $(E_{r}^{p,q}, d_r)$ be the weight spectral sequence for $X\hookrightarrow {\Xcpt}$. 
Since $H^1(X)$ is pure, we have $E_{2}^{-1,2}={\Gr}^{W}_{2}H^1(X)=0$ by \eqref{eqn: weight ss E2}. 
This means that the complex 
$E_{1}^{-2,2}\to E_{1}^{-1,2}\to E_{1}^{0,2}$ 
is exact. 
By \eqref{eqn: weight E1}, this complex is written as 
$0\to H^0(D(1))\to H^2({\Xcpt})$. 
This shows that the Gysin map 
\begin{equation*}
H^0(D(1)) = \bigoplus_{i=1}^{\alpha} {\C}D_{i} \to H^2({\Xcpt}) 
\end{equation*}
is injective. 
Here note that the sign $(-1)^{j-1}$ in \eqref{eqn: Gysin sign} is $1$. 
The cokernel of this map $E_{1}^{-1,2}\to E_{1}^{0,2}$ is 
\begin{equation*}
E_{2}^{0,2} = {\Gr}^W_2H^2(X) = W_2H^2(X) =H^2(X), 
\end{equation*}
where the last equality follows from the purity of $H^2(X)$. 

(2) By the purity of $H^2(X)$, we have 
$E_{2}^{-1,3} = {\Gr}_{3}^{W}H^2(X) = 0$ 
by \eqref{eqn: weight ss E2}. 
Hence the complex $E_{1}^{-2,3} \to E_{1}^{-1,3} \to E_{1}^{0,3}$ is exact. 
By \eqref{eqn: weight E1}, this complex is written as 
$0 \to H^1(D(1)) \to H^3({\Xcpt})$. 
It follows that the Gysin map $H^1(D(1)) \to H^3({\Xcpt})$ is injective. 
\end{proof}

In the case of Siegel modular varieties, 
linear independence of the boundary divisors was proved by Hoffman-Weintraub \cite{HW2} 
by using the vanishing of $H^1(X)$. 
Our proof uses instead the purity of $H^1(X)$. 

\section{Weight filtration and corank filtration}\label{sec: weight Siegel}

Let $X={\D}/{\G}$ be a locally symmetric variety as in \S \ref{sec: toroidal} with ${\G}$ neat. 
In \S \ref{sec: VT II}, 
we proved that the weight filtration on $F^kH^k(X)$ is trivial when $k<n$. 
When $k=n$, this is no longer true, and this case is in fact more interesting. 
This is our object of study in \S \ref{sec: weight Siegel} and \S \ref{sec: GrW}, 
which are the central part of this paper. 

Throughout \S \ref{sec: weight Siegel} and \S \ref{sec: GrW}, we use the abbreviations 
$F^n=F^nH^n(X)$ and $W_{m}F^{n}=W_{m}H^n(X) \cap F^{n}H^{n}(X)$. 
We denote by ${\Gr}_m^{W}F^n=W_{m}F^{n}/W_{m-1}F^{n}$ 
the $m$-th graded quotient of this weight filtration on $F^n$.  
By Theorem \ref{thm: VT I}, we have ${\Gr}^{W}_{n+m}F^{n}=0$ 
unless $m=n(i)$ for some $0\leq i \leq r$. 
Hence the weight filtration on $F^n$ reduces to length $r$: 
\begin{equation}\label{eqn: weight filtration Fn}
W_nF^n \subset W_{n+n(1)}F^n \subset W_{n+n(2)}F^n \subset 
\cdots \subset W_{n+n(r)}F^n = F^n. 
\end{equation}
In this section we relate this filtration with modular forms. 
In \S \ref{ssec: corank filtration}, we define the ``corank filtration'' on the space of 
modular forms of canonical weight by using the Siegel operators. 
In \S \ref{ssec: weight=corank}, we prove that \eqref{eqn: weight filtration Fn} coincides with the corank filtration. 
A key step is construction of the Shioda-type maps \eqref{eqn: construct piF}. 
In \S \ref{ssec: GF monodromy} and \S \ref{ssec: GF invariant}, 
we present a restriction on their image. 
This is essentially a preliminary for the next \S \ref{sec: GrW}. 

Throughout \S \ref{sec: weight Siegel} and \S \ref{sec: GrW}, 
we fix a smooth projective toroidal compactification ${\Xcpt}$ of $X$ with SNC boundary divisor $D={\Xcpt}-X$,  
together with a numbering of the irreducible components of $D$. 
The latter is necessary for applying mixed Hodge theory. 
In particular, it gets rid of the ambiguity of sign when taking residue at strata of codimension $>1$.

\subsection{Corank filtration}\label{ssec: corank filtration}

Let ${\LL}$ be the automorphic line bundle of canonical weight extended over $X^{bb}$ (\S \ref{ssec: BB}). 
We denote by 
\begin{equation*}
M_{can}({\G}) := H^0(X^{bb}, {\LL}) = H^0(X, K_{X}) = H^0({\D}, K_{{\D}})^{{\G}} 
\end{equation*}
the space of ${\G}$-modular forms of canonical weight. 
Here the second equality holds because the boundary of $X^{bb}$ is assumed to be of codimension $\geq 2$. 
Let $S\!_{can}({\G})$ be the subspace of cusp forms, i.e., 
sections of ${\LL}$ over $X^{bb}$ which vanish at the boundary of $X^{bb}$. 
We often abbreviate $M_{can}=M_{can}({\G})$ and $S\!_{can} = S\!_{can}({\G})$ 
when the group ${\G}$ is clear from the context. 

The restriction ${\LL}|_{X_{F}}$ of ${\LL}$ to each cusp $X_{F}\subset X^{bb}$ 
is the descent of an automorphic line bundle on $F$ (\cite{BB}). 
Moreover, we can restrict ${\LL}$ to the closure $\overline{X}_{F}$ of $X_{F}$ in $X^{bb}$. 
The pullback of ${\LL}|_{\overline{X}_{F}}$ by the normalization map $X_{F}^{bb}\to \overline{X}_{F}$ 
gives an extension of ${\LL}|_{X_{F}}$ to the Baily-Borel compactification $X_{F}^{bb}$ of $X_{F}$ itself.  
We denote this line bundle on $X_{F}^{bb}$ by ${\LL}_{F}$. 
Then we write  
\begin{equation*}
M_{cat}(F) = M_{cat}(F, {\GFZ}) := H^0(X_{F}^{bb}, {\LL}_{F}), 
\end{equation*}
and denote by $S\!_{cat}(F)\subset M_{cat}(F)$ the subspace of cusp forms, 
i.e., sections vanishing at the boundary $X_{F}^{bb}-X_{F}$ of $X_{F}^{bb}$. 
Pulling back the sections of ${\LL}$ by $X_{F}^{bb}\to \overline{X}_{F}\subset X^{bb}$ 
defines a linear map  
\begin{equation*}
\Phi_{F} : M_{can} \to M_{cat}(F). 
\end{equation*}
This restriction operator is usually called the \textit{Siegel operator} at $F$. 

\begin{remark}
The notation $cat$ is an abbreviation of "canonical total", 
by which we intend to indicate that it corresponds to the canonical bundle of  
the total space of an abelian fibration over $F$ (cf.~ \cite{BB} \S 8.2). 
\end{remark}

For each $1\leq i \leq r$, 
we denote by $\mathcal{C}(i)$ the set of ${\G}$-equivalence classes of cusps $F$ of ${\D}$ with $\dim U(F)=n(i)$. 
For $0\leq i \leq r$, 
we define a subspace ${\Mi}$ of $M_{can}$ as the kernel of the total Siegel operator  
\begin{equation*}
(\Phi_{F})_{F} : M_{can} \to \bigoplus_{\substack{F\in \mathcal{C}(j) \\ j>i}} M_{cat}(F) 
\end{equation*}
to all cusps in $\mathcal{C}(j)$ with $j>i$. 
Thus ${\Mi}$ consists of modular forms which vanish at all cusps $F$ with $\dim U(F)>n(i)$. 
We have the filtration 
\begin{equation}\label{eqn: corank filtration}
S\!_{can} = M_{can}^{(0)} \subset M_{can}^{(1)} \subset M_{can}^{(2)} \subset \cdots \subset M_{can}^{(r)} = M_{can}. 
\end{equation}
If $F\in \mathcal{C}(i)$, the Siegel operator $\Phi_{F}$ sends modular forms in ${\Mi}$ to cusp forms on $F$: 
\begin{equation}\label{eqn: cuspidal Siegel}
\Phi_{F} : {\Mi} \to S\!_{cat}(F). 
\end{equation}
Indeed, cusps $F'$ with $F'\prec F$ satisfy $U(F)\subsetneq U(F')$ and so $\dim U(F')>\dim U(F)=n(i)$. 
Hence modular forms in ${\Mi}$ vanish at $F'$.  
We call the map \eqref{eqn: cuspidal Siegel} the \textit{cuspidal Siegel operator} at $F$. 
Each step $M_{can}^{(i-1)}\subset {\Mi}$ in the filtration \eqref{eqn: corank filtration} is given by 
\begin{equation}\label{eqn: corank filtration successive}
M_{can}^{(i-1)} = \ker \left( (\Phi_{F})_{F} : 
{\Mi} \longrightarrow \bigoplus_{F\in \mathcal{C}(i)}S\!_{cat}(F) \right). 
\end{equation}

When the algebraic group $\mathbb{G}$ is ${\Q}$-simple, 
it would be natural to call \eqref{eqn: corank filtration} the \textit{corank filtration}. 
In this case, every cusp of corank $>i+1$ is in the closure of a cusp of corank $i+1$ (\cite{BB} Theorem 3.8), 
so we can define ${\Mi}$ more simply as 
\begin{equation*}
{\Mi} = \ker \left( (\Phi_{F})_{F} : M_{can} \longrightarrow \bigoplus_{F\in \mathcal{C}(i+1)} M_{cat}(F) \right). 
\end{equation*}

\subsection{Comparison with weight filtration}\label{ssec: weight=corank}

In this subsection we compare the weight filtration \eqref{eqn: weight filtration Fn} on $F^n$ with 
the filtration \eqref{eqn: corank filtration} on $M_{can}$. 
Recall from \eqref{eqn: FnHn K+D} that 
$F^n=F^nH^n(X)$ is naturally isomorphic to $H^0({\Xcpt}, K_{{\Xcpt}}(D))$. 
Combining this with the Mumford isomorphism $\pi^{\ast}{\LL}\simeq K_{{\Xcpt}}(D)$ (see \eqref{eqn: L = K+D}), 
we obtain a canonical isomorphism 
\begin{equation*}
\pi^{\ast} :  M_{can}({\G}) \stackrel{\simeq}{\to} 
H^0({\Xcpt}, \pi^{\ast}{\LL}) \simeq H^0({\Xcpt}, K_{{\Xcpt}}(D)) \simeq F^nH^n(X). 
\end{equation*}


\begin{theorem}\label{thm: weight=corank}
The isomorphism 
$\pi^{\ast}\colon M_{can} \simeq F^n$ 
identifies the filtration \eqref{eqn: corank filtration} with the weight filtration \eqref{eqn: weight filtration Fn}. 
Moreover, for each cusp $F$ in $\mathcal{C}(i)$ and a top-dimensional cone $\sigma$ in $\Sigma_{F}$, 
we have the commutative diagram 
\begin{equation}\label{eqn: residue=Siegel}
\xymatrix{
W_{n+n(i)}F^n \ar[r]^-{{\res}_{F,\sigma}}  & H^0(K_{{\DFsbar}})  \\ 
M_{can}^{(i)} \ar[u]^{\pi^{\ast}}_{\simeq} \ar[r]^{\Phi_{F}} & S\!_{cat}(F) \ar@{^{(}-_>}[u]_{\pi_{F,\sigma}^{\ast}} 
}
\end{equation}
where ${\res}_{F,\sigma}$ is the residue map at ${\DFsbar}$ 
and $\pi_{F,\sigma}^{\ast}$ is the injective map constructed in \eqref{eqn: construct piF} below. 
\end{theorem}

As a preliminary for the proof of Theorem \ref{thm: weight=corank}, we first note the following. 

\begin{lemma}\label{lem: residue only corank i}
Let $1\leq i \leq r$. 
Then $W_{n+n(i-1)}F^n$ is the kernel of the residue map 
\begin{equation*}
W_{n+n(i)}F^n \to \bigoplus_{(F,\sigma)} H^0(K_{{\DFsbar}}), 
\end{equation*}
where $(F, \sigma)$ ranges over ${\G}$-equivalence classes of pairs of 
a cusp $F$ with $\dim U(F)=n(i)$ and a top-dimensional cone $\sigma$ in $\Sigma_{F}$. 
\end{lemma}

\begin{proof}
By Theorem \ref{thm: VT I}, we have $W_{n+n(i-1)}F^n= W_{n+n(i)-1}F^n$. 
By \eqref{eqn: WmFk residue kernel} and the description \eqref{eqn: D(m) toroidal} of $D(m)$,  
we see that $W_{n+n(i)-1}F^n$ is the kernel of the residue map
\begin{equation}\label{eqn: residue map weight n+n(i)}
W_{n+n(i)}F^n \to \bigoplus_{\substack{(F,\sigma) \\ \dim \sigma = n(i)}} H^0(K_{{\DFsbar}}), 
\end{equation}
where $(F, \sigma)$ ranges over all pairs with $\sigma\in \Sigma_{F}^{\circ}$ and $\dim \sigma = n(i)$ 
up to ${\G}$-equivalence. 
If $\dim U(F) < n(i)$, then $\dim \sigma < n(i)$ for any $\sigma\in \Sigma_{F}$, 
so such cusps $F$ do not appear in the target of \eqref{eqn: residue map weight n+n(i)}. 
If $\dim U(F) > \dim \sigma = n(i)$, 
the torus bundle ${\DFs}\to Y_{F}$ has positive-dimensional fibers. 
Then ${\DFsbar}$ is uniruled, so $H^0(K_{{\DFsbar}})$ vanishes. 
Therefore only pairs $(F, \sigma)$ with $\dim U(F) = \dim \sigma = n(i)$ 
contribute to the target of \eqref{eqn: residue map weight n+n(i)}. 
\end{proof}

For the proof of Theorem \ref{thm: weight=corank}, 
we need to construct the injective map 
$S\!_{cat}(F) \hookrightarrow H^0(K_{{\DFsbar}})$ 
in \eqref{eqn: residue=Siegel}. 
This requires a few steps below, 
and the construction will be completed in \eqref{eqn: construct piF}.  

Let $F$ be a cusp in $\mathcal{C}(i)$ and $\sigma$ be a top-dimensional cone in $\Sigma_{F}$.  
The projection $\pi \colon {\Xcpt}\to X^{bb}$ gives a surjective morphism ${\DFsbar}\to \overline{X}_{F}$. 
Since $X_{F}^{bb}$ is the normalization of $\overline{X}_{F}$, 
this lifts to a morphism to $X_{F}^{bb}$ by the universal property of normalization. 
We denote it by 
\begin{equation*}
\pi_{F,\sigma}\colon {\DFsbar}\to X_{F}^{bb}.  
\end{equation*}
By construction we have the commutative diagram 
\begin{equation}\label{CD: DFsbar XFbb}
\xymatrix{
{\DFsbar} \ar@{^{(}-_>}[r]  \ar[d]_{\pi_{F,\sigma}} & {\Xcpt}  \ar[d]^{\pi} \\ 
X_{F}^{bb} \ar[r] & X^{bb} 
}
\end{equation}

Let ${\DFs}'={\DFsbar}-{\DFs}$ be the boundary divisor of ${\DFsbar}$ itself. 
By Lemma \ref{lem: boundary DFs}, this is the intersection of ${\DFsbar}$ with 
the union of the irreducible components of $D$ that do not contain ${\DFs}$. 
By the log adjunction formula \eqref{eqn: log adjunction general} restricted to a connected component, 
we have 
\begin{equation}\label{eqn: log adjunction}
K_{{\Xcpt}}(D)|_{{\DFsbar}} \simeq K_{{\DFsbar}}({\DFs}'). 
\end{equation}
Thus we have the natural isomorphism 
\begin{equation}\label{eqn: KYF}
\pi_{F,\sigma}^{\ast}{\LL}_{F} \simeq (\pi^{\ast}{\LL})|_{{\DFsbar}} \simeq K_{{\DFsbar}}({\DFs}'). 
\end{equation}
Then the pullback by $\pi_{F,\sigma}$ and the isomorphism \eqref{eqn: KYF} define a linear map 
\begin{equation}\label{eqn: construct piF larger}
\pi_{F,\sigma}^{\ast} : M_{cat}(F) = H^0(X_{F}^{bb},  {\LL}_{F}) 
\to H^0({\DFsbar}, \: K_{{\DFsbar}}({\DFs}')). 
\end{equation}
Since $\pi_{F,\sigma}$ is surjective, the pullback map $\pi_{F,\sigma}^{\ast}$ is injective. 

We shall show that $\pi_{F,\sigma}^{\ast}(S\!_{cat}(F))$ is contained in $H^0(K_{{\DFsbar}})$. 
Since $\sigma$ is top-dimensional in $\Sigma_{F}$, 
every stratum $D(F', \sigma')$ of ${\DFs}'$ in \eqref{eqn: closure stratum} 
satisfies $F'\prec F$. 
Thus $\pi_{F,\sigma}$ maps ${\DFs}'$ to the boundary of $X_{F}^{bb}$.
If $\mathcal{I}$ is the ideal sheaf of the boundary of $X_{F}^{bb}$, 
this implies that $\pi_{F,\sigma}^{\ast}\mathcal{I}$ is contained in the ideal sheaf of ${\DFs}'$. 
Therefore we have  
\begin{equation*}
\pi_{F,\sigma}^{\ast}(\mathcal{I} \! \cdot \! {\LL}_{F}) \stackrel{\eqref{eqn: KYF}}{\hookrightarrow} 
\mathcal{O}_{{\DFsbar}}(-{\DFs}') \cdot K_{{\DFsbar}}({\DFs}') = K_{{\DFsbar}}. 
\end{equation*}
In this way, as a restriction of \eqref{eqn: construct piF larger}, we obtain an injective map 
\begin{equation}\label{eqn: construct piF}
\pi_{F,\sigma}^{\ast} : S\!_{cat}(F) = H^0(X_{F}^{bb}, \: \mathcal{I}\! \cdot \! {\LL}_{F}) 
\hookrightarrow H^0({\DFsbar}, \: K_{{\DFsbar}}). 
\end{equation}

We can now proceed to the proof of Theorem \ref{thm: weight=corank}. 

\begin{proof}[(Proof of Theorem \ref{thm: weight=corank})]
We proceed inductively on $i$, starting from $i=r$ and decreasing as follows. 
The first step $i=r$ is just the isomorphism $\pi^{\ast}\colon M_{can}\simeq F^n$ itself. 
Suppose that we have proved that 
$\pi^{\ast}\colon M_{can}\to F^n$ maps ${\Mi}$ to $W_{n+n(i)}F^n$. 
Then we will prove that the diagram \eqref{eqn: residue=Siegel} commutes 
for each $(F, \sigma)$. 
This commutativity and the injectivity of 
$S\!_{cat}(F) \hookrightarrow H^0(K_{{\DFsbar}})$ 
imply that the isomorphism 
$\pi^{\ast}\colon {\Mi} \to W_{n+n(i)}F^n$ maps $\ker(\Phi_{F})$ to $\ker({\res}_{F,\sigma})$. 
This shows in particular that $\ker({\res}_{F,\sigma})$ does not depend on $\sigma$. 
Taking the intersection over all cusps in $\mathcal{C}(i)$, 
we see that the following equalities hold in 
$W_{n+n(i)}F^n=\pi^{\ast}({\Mi})$: 
\begin{equation*}
\pi^{\ast}(M_{can}^{(i-1)}) = \pi^{\ast}\left( \bigcap_{F\in \mathcal{C}(i)}\ker(\Phi_{F}) \right) 
= \bigcap_{F\in \mathcal{C}(i)}\ker({\res}_{F,\sigma}) = W_{n+n(i-1)}F^n. 
\end{equation*}
Here the first equality follows from \eqref{eqn: corank filtration successive}, 
and the last equality follows from Lemma \ref{lem: residue only corank i}. 
In this way we could proceed from step $i$ to step $i-1$. 
It thus remains to verify the commutativity of \eqref{eqn: residue=Siegel} in step $i$, 
assuming that we have proved that $\pi^{\ast}\colon M_{can}\to F^n$ maps ${\Mi}$ to $W_{n+n(i)}F^n$. 

We shall prove that the larger diagram 
\begin{equation}\label{eqn: residue=Siegel larger}
\xymatrix@C+1pc{
F^n \ar[r]^-{{\res}_{F,\sigma}}  & H^0(K_{{\DFsbar}}({\DFs}'))  \\ 
M_{can} \ar[u]^{\pi^{\ast}}_{\simeq} \ar[r]^{\Phi_{F}} & M_{cat}(F) \ar@{^{(}-_>}[u]_{\pi_{F,\sigma}^{\ast}} 
}
\end{equation}
commutes. 
The commutativity of \eqref{eqn: residue=Siegel} then follows by 
restricting this diagram to subspaces. 
Let $\omega\in M_{can}$. 
We consider the commutative diagram \eqref{CD: DFsbar XFbb}. 
Since the Siegel operator $\Phi_{F}$ is the pullback from $X^{bb}$ to $X_{F}^{bb}$ 
and $\pi_{F,\sigma}^{\ast}$ is given by 
the pullback by $\pi_{F,\sigma}$ plus the isomorphism \eqref{eqn: KYF}, 
we see that $\pi_{F,\sigma}^{\ast}(\Phi_{F}(\omega))$ 
is the restriction of the log canonical form $\pi^{\ast}\omega$ to ${\DFsbar}\subset {\Xcpt}$ via 
the adjunction isomorphism \eqref{eqn: log adjunction}. 
Since the adjunction \eqref{eqn: log adjunction} is given by 
taking the residues of log canonical forms along ${\DFsbar}$, we obtain 
$\pi_{F,\sigma}^{\ast}(\Phi_{F}(\omega)) = {\res}_{F,\sigma}(\pi^{\ast}\omega)$. 
This means that \eqref{eqn: residue=Siegel larger} commutes. 
\end{proof}

The diagram \eqref{eqn: residue=Siegel larger} can be explicitly understood 
at the level of ${\SFs}\subset {\XFcpt}$ as follows. 
Let $(z_j, u_k, t_l)$ be a Siegel domain coordinate on ${\D}\subset {\D}(F)$ as in \cite{AMRT} \S IV.1.1. 
After a change of variables for $z_{j}$, 
we may assume that ${\SFs}$ is defined by $q_1=\cdots = q_{m}=0$ 
where $q_{j}=\exp (2\pi i z_j)$ and $m=n(i)$. 
We abbreviate $dz_{j}=\wedge_{j}dz_{j}$ etc. 
Then we can express $\omega\in M_{can}$ as 
\begin{equation*}
\omega = f(z_j, u_k, t_l) dz_j \wedge du_{k} \wedge dt_l  =  
\frac{1}{(2\pi i)^{m}} f(z_j, u_k, t_l) \frac{dq_j}{q_j} \wedge du_{k} \wedge dt_l. 
\end{equation*}
Hence the residue of $\omega$ at ${\SFs}\simeq {\VF}$ is given by 
$f(i\infty, u_k, t_l) du_{k} \wedge dt_l$. 
On the other hand, after trivializing $K_{{\D}}$ by the frame $dz_j \wedge du_{k} \wedge dt_l$, 
the Siegel operator at $F$ sends $f$ to $f(i\infty, u_k, t_l)$, 
which does not depend on $u_{k}$ (cf.~\cite{BB} \S 8.2 -- \S 8.5). 
The map $\pi_{F,\sigma}^{\ast}$ recovers the hidden frame $du_{k} \wedge dt_l$ and 
sends this section to the canonical form 
$f(i\infty, u_k, t_l)du_{k} \wedge dt_l$ on ${\VF}$. 
A prototype of such a correspondence can be found in the work of Shioda \cite{Shioda} on elliptic modular surfaces.

\subsection{Monodromy action of ${\GF}$}\label{ssec: GF monodromy}

In the next \S \ref{ssec: GF invariant}, we will show that the image of the map \eqref{eqn: construct piF} 
satisfies invariance under the ``monodromy action'' of ${\GF}$. 
In this \S \ref{ssec: GF monodromy}, as a preliminary, let us first explain this action. 

Let $F$ be a cusp of ${\D}$. 
Recall from \S \ref{ssec: stratification} that a local model around the $F$-stratum $\pi^{-1}(X_{F})$ in ${\Xcpt}$ 
is provided by ${\XFcpt}/{\GFZbar}$. 
We consider the intermediate quotient 
\begin{equation*}
{\YFcpt} := {\XFcpt}/{\GFZbd}. 
\end{equation*}
Since ${\GFZbd}$ is normal in ${\GFZbar}$, 
the quotient group ${\GF}$ acts on ${\YFcpt}$. 
The natural projection ${\XFcpt}\to {\VF}$, being ${\GFZbar}$-equivariant, 
descends to a ${\GF}$-equivariant map ${\YFcpt}\to Y_{F}$. 
The situation is summarized as  
\begin{equation*}
\xymatrix@C+1pc{
{\XFcpt} \ar[r]^{/{\GFZbd}}   \ar[d] & {\YFcpt}  \ar[d] \ar[r]^{/{\GF}} & {\XFcpt}/{\GFZbar}  \ar[r] & {\Xcpt} \\ 
{\VF} \ar[r]^{/{\GFZbd}}  & Y_{F}  
}
\end{equation*}

For an $F$-cone $\sigma\in \Sigma_{F}^{\circ}$, 
we denote by $\hat{D}(F, \sigma)$ the image of ${\SFs}$ in ${\YFcpt}$. 
Then $\hat{D}(F, \sigma)\simeq {\SFs}/{\GFZbd}$. 
Although the natural map $\hat{D}(F, \sigma) \to {\DFs}$ is an isomorphism, 
it is important to distinguish its source and target. 
The group ${\GF}$ permutes the strata $\hat{D}(F, \sigma)$ according to its free action on $\Sigma_{F}^{\circ}$. 
We denote by $p_{\sigma}\colon \hat{D}(F, \sigma)\to Y_{F}$ the restriction of ${\YFcpt}\to Y_{F}$ to $\hat{D}(F, \sigma)$. 
This is a principal torus bundle for the quotient torus of $T(F)$ defined by $\sigma$. 
If we let $\gamma\in {\GF}$ act, we have the commutative diagram 
\begin{equation}\label{CD: gamma act on Dhat}
\xymatrix{
\hat{D}(F, \sigma) \ar[r]^{\gamma}  \ar[d]_{p_{\sigma}} & \hat{D}(F, \gamma \sigma) \ar[d]^{p_{\gamma\sigma}}  \\ 
Y_{F} \ar[r]^{\gamma} & Y_{F}  
}
\end{equation}

Now let $\sigma\in \Sigma_{F}$ be a top-dimensional cone. 
In this case $p_{\sigma}$ is an isomorphism. 
Hence the $\gamma$-action on $Y_{F}$ is identified with the $\gamma$-action 
$\hat{D}(F, \sigma)\to \hat{D}(F, \gamma \sigma)$ inside ${\YFcpt}$ 
via $p_{\sigma}$ and $p_{\gamma\sigma}$. 
On the other hand, we also have a ``monodromy isomorphism'' $\hat{D}(F, \sigma)\to \hat{D}(F, \gamma \sigma)$ as follows. 

Let $\sigma_{+}, \sigma_{-}$ be two top-dimensional cones in $\Sigma_{F}$ 
sharing a codimension $1$ face $\tau$. 
The projection 
\begin{equation}\label{eqn: P1 bundle}
p_{\sigma_{+}} \sqcup p_{\tau} \sqcup p_{\sigma_{-}} \: : \: 
\hat{D}(F, \sigma_{+}) \sqcup \hat{D}(F, \tau) \sqcup \hat{D}(F, \sigma_{-}) \to Y_{F} 
\end{equation}
is a ${\proj}^1$-bundle over $Y_{F}$ 
with $\hat{D}(F, \sigma_{+})$ the $0$-section and $\hat{D}(F, \sigma_{-})$ the $\infty$-section. 
This ruling structure defines a natural isomorphism 
$\hat{D}(F, \sigma_{+})\to \hat{D}(F, \sigma_{-})$ 
which coincides with $p_{\sigma_{-}}^{-1} \circ p_{\sigma_{+}}$. 

We connect $\sigma$ and $\gamma \sigma$ by a chain of such adjacent cones. 
This defines an isomorphism 
\begin{equation*}
\Psi: \hat{D}(F, \sigma)\to \hat{D}(F, \gamma \sigma) 
\end{equation*}
which satisfies $p_{\gamma\sigma} \circ \Psi = p_{\sigma}$. 
(This equality characterizes $\Psi$ uniquely.) 
The $\gamma$-action on $Y_{F}$, now translated to an automorphism of $\hat{D}(F, \sigma)$ via $p_{\sigma}$, 
can be understood as the difference of the $\gamma$-action $\hat{D}(F, \sigma)\to \hat{D}(F, \gamma \sigma)$ 
and the monodromy isomorphism $\Psi$.  

Finally, we come to a matter of orientation. 
According to Lemma \ref{lem: D(m) toroidal}, 
the numbering of the irreducible components of ${\Xcpt}-X$ 
defines a numbering of the rays of $\sigma$ and of $\gamma\sigma$. 
This determines an orientation of $\sigma$ and $\gamma\sigma$. 
(Let us call it the MHS orientation.) 
Since $\sigma$ and $\gamma\sigma$ are top-dimensional cones, 
we can compare their orientation via that of $U(F)$. 
Then the following holds. 

\begin{lemma}\label{lem: orientation}
$\sigma$ and $\gamma\sigma$ have the same orientation. 
\end{lemma}

\begin{proof}
By construction, $\gamma \colon \sigma \to \gamma \sigma$ 
preserves the MHS orientation. 
On the other hand, since ${\GF}$ is neat, 
$\gamma \colon U(F) \to U(F)$ 
preserves the orientation of $U(F)$. 
This implies our assertion. 
\end{proof}

\subsection{Monodromy invariance of the Shioda-type map}\label{ssec: GF invariant}
 
Let $[\sigma]\in \Sigma_{F}/{\GF}$ be a ${\GF}$-equivalence class of top-dimensional cones.  
We go back to the map 
$\pi_{F,\sigma}^{\ast}\colon S\!_{cat}(F)\to H^0(K_{{\DFsbar}})$ 
constructed in \eqref{eqn: construct piF}. 
We choose a cone $\sigma\in \Sigma_{F}$ representing $[\sigma]$. 
This enables us to take the isomorphism 
\begin{equation*}
{\DFs} \simeq \hat{D}(F, \sigma) \stackrel{p_{\sigma}}{\to} Y_{F}, 
\end{equation*}
which we again denote by $p_{\sigma}$. 
We denote by 
\begin{equation*}
\pi_{F}^{\ast} : S\!_{cat}(F)\to H^0(K_{\overline{Y}_{F}}) 
\end{equation*}
the composition of $\pi_{F,\sigma}^{\ast}$ and $p_{\sigma \ast}$ 
where $\overline{Y}_{F}$ is an arbitrary smooth projective compactification of $Y_{F}$. 
Then the following holds. 

\begin{proposition}\label{prop: Shioda monodromy}
The map $\pi_{F}^{\ast}$ does not depend on the choice of the cone $\sigma$ representing $[\sigma]\in \Sigma_{F}/{\GF}$, 
and its image is contained in $H^0(K_{\overline{Y}_{F}})^{{\GF}}$. 
\end{proposition}

\begin{proof}
By the diagram \eqref{CD: gamma act on Dhat}, 
the ${\GF}$-invariance follows from the independence from $\sigma$. 
We prove the latter. 
Since $\pi_{F}^{\ast}$ is defined locally, 
it suffices to check the independence for local sections of $\mathcal{L}|_{X_{F}}$. 

Let $U\subset X^{bb}$ be a small open neighborhood of a point of $X_{F}$. 
Let $\omega$ be a section of $\mathcal{L}|_{X_{F}}$ over $U\cap X_{F}$. 
We can extend $\omega$ to a section of $\mathcal{L}$ over $U$, 
which is identified with a canonical form on $U\cap X$. 
We take its pullback to ${\YFcpt}$. 
This is a ${\GF}$-invariant log canonical form $\tilde{\omega}$ on 
a neighborhood of the inverse image $V\subset {\YFcpt}$ of $U\cap X_{F}$. 
Then we have 
\begin{equation*}
\pi_{F}^{\ast}(\omega) = p_{\sigma \ast} \pi_{F,\sigma}^{\ast}(\omega) = 
p_{\sigma \ast}({\res}_{\sigma}(\tilde{\omega})), 
\end{equation*}
where ${\res}_{\sigma}$ means the residue at $V\cap \hat{D}(F, \sigma)$. 
Since the monodromy isomorphism 
$\Psi \colon \hat{D}(F, \sigma) \to \hat{D}(F, \gamma\sigma)$ 
satisfies $p_{\gamma \sigma} \circ \Psi = p_{\sigma}$, 
what has to be shown is 
\begin{equation}\label{eqn: monodromy invariance residue}
{\res}_{\sigma}(\tilde{\omega}) = \Psi^{\ast}({\res}_{\gamma \sigma}(\tilde{\omega})). 
\end{equation}
Here the residues at $\hat{D}(F, \sigma)$ and $\hat{D}(F, \gamma \sigma)$ 
are taken with respect to the MHS orientation of $\sigma$ and $\gamma\sigma$ respectively. 
By Lemma \ref{lem: orientation}, they come from a common orientation of $U(F)$. 
This in turn induces an orientation of every top-dimensional cone in $\Sigma_{F}$. 
With this orientation system for taking residues, we shall prove 
\begin{equation}\label{eqn: residue independent}
{\res}_{\sigma_{+}}(\tilde{\omega}) = {\res}_{\sigma_{-}}(\tilde{\omega}) 
\end{equation}
for any two adjacent top-dimensional cones $\sigma_{+}, \sigma_{-}$ sharing a codimension $1$ face $\tau$ 
as in \eqref{eqn: P1 bundle}. 
Our assertion \eqref{eqn: monodromy invariance residue} then follows by connecting $\sigma$ and $\gamma\sigma$ 
by a chain of such cones. 

With the orientation of $\tau$ induced by that of $\sigma_{+}$, 
we take the residue of $\tilde{\omega}$ at $V\cap \hat{D}(F, \tau)$. 
This is a log canonical form on the restriction of the ${\proj}^1$-bundle \eqref{eqn: P1 bundle} over an open set of $Y_{F}$, 
with pole at $\hat{D}(F, \sigma_{+})$ and $\hat{D}(F, \sigma_{-})$. 
Taking its residue at $0\in {\proj}^1$ gives ${\res}_{\sigma_{+}}(\tilde{\omega})$, 
while at $\infty \in {\proj}^1$ gives $-{\res}_{\sigma_{-}}(\tilde{\omega})$. 
Here the sign gets reversed because the orientation of $\tau$ induced by that of $\sigma_{-}$ 
is opposite to the one induced by $\sigma_{+}$. 
Since ${\res}_{0}(dz/z)=-{\res}_{\infty}(dz/z)$ 
for the standard log canonical form $dz/z$ on ${\proj}^1$, 
the sign change cancels and we obtain \eqref{eqn: residue independent}. 
\end{proof}

In \S \ref{sec: GrW}, we will show that 
${\rm Im}(\pi_{F}^{\ast}) = H^0(K_{\overline{Y}_{F}})^{{\GF}}$ 
in most cases.

\section{Weight graded quotients}\label{sec: GrW}

We keep the setting and notation of \S \ref{sec: weight Siegel}. 
In this section we calculate the graded quotients of the weight filtration 
$W_{\bullet}F^{n}$ on $F^{n}=F^{n}H^n(X)$ 
and give an application to the Siegel operators.

\subsection{Main results}\label{ssec: main GrW}

We first state our main result (Theorem \ref{thm: GrWFn}) and derive its immediate consequences, 
postponing the proof to the next \S \ref{ssec: proof GrW}. 
Let $1\leq i \leq r$. 
We calculate 
$W_{n+n(i)}F^n/W_{n+n(i-1)}F^n$ 
under the condition 
\begin{equation}\label{eqn: n(i) condition}
n(i) - n(i-1) > 1. 
\end{equation} 
This condition implies that 
if $F\in \mathcal{C}(i)$ and $F'\succ F$, then 
$U(F')$ is of codimension $\geq 2$ in $U(F)$. 
Therefore no codimension $1$ cone in $\Sigma_{F}$ is contained in the boundary 
$C(F)^{\ast}-C(F)$ of $C(F)$. 
In other words: 
\begin{equation}\label{eqn: codim 1 F-cone}
\textrm{If $F\in \mathcal{C}(i)$, every codimension $1$ cone $\tau$ in $\Sigma_{F}$ is an $F$-cone.}
\end{equation}
In particular, for such a cone $\tau$, 
there are exactly two top-dimensional cones $\sigma_{+}, \sigma_{-}\in \Sigma_{F}$ 
which have $\tau$ as their face. 
We will use the condition \eqref{eqn: n(i) condition} in this way.

\begin{theorem}\label{thm: GrWFn}
Let $1\leq i \leq r$ and assume that the condition \eqref{eqn: n(i) condition} holds. 
Then we have an isomorphism 
\begin{equation*}
W_{n+n(i)}F^n/W_{n+n(i-1)}F^n \to 
\bigoplus_{F\in \mathcal{C}(i)} H^0(K_{\overline{Y}_{F}})^{{\GF}} 
\end{equation*}
induced by the residue maps at ${\DFs}\simeq Y_{F}$ where $\sigma$ is an arbitrary top-dimensional cone in $\Sigma_{F}$. 
\end{theorem}

A special case of Theorem \ref{thm: GrWFn}  
generalizes some results of Oda \cite{Od}, Ziegler \cite{Fr} and Miyazaki \cite{Mi}: 

\begin{corollary}\label{cor: 0-dim cusp}
Assume that $n(r)=n$, i.e., ${\D}$ has a tube domain realization, and $n(r-1)<n-1$. 
Then $\dim {\Gr}^{W}_{2n}H^n(X)$ is equal to the number of $0$-dimensional cusps of $X^{bb}$. 
\end{corollary}

\begin{proof} 
Since ${\Gr}^{W}_{2n}H^n(X)={\Gr}^{W}_{2n}F^nH^n(X)$ and 
$Y_{F}=X_{F}$ are points for $F\in \mathcal{C}(r)$, we have 
$\dim {\Gr}^{W}_{2n}H^n(X) = |\mathcal{C}(r)|$ 
by Theorem \ref{thm: GrWFn}. 
Since $0$-dimensional boundary components of ${\D}$ are 
equivalent to each other under ${\rm Aut}({\D})^{\circ}$, 
we have $\dim U(F)=n$ for any $0$-dimensional cusp $F$. 
Hence $\mathcal{C}(r)$ is the set of $0$-dimensional cusps of $X^{bb}$. 
\end{proof}

Let 
\begin{equation*}
\Phi^{(i)}_{0} =(\Phi_{F})_{F} : M_{can}^{(i)} \to \bigoplus_{F\in \mathcal{C}(i)}S\!_{cat}(F) 
\end{equation*}
be the direct sum of the cuspidal Siegel operators to all cusps in $\mathcal{C}(i)$. 
Combining Theorem \ref{thm: GrWFn} with the results of \S \ref{sec: weight Siegel}, 
we obtain the following. 

\begin{theorem}\label{thm: CD}
Let $1\leq i \leq r$ and assume that the condition \eqref{eqn: n(i) condition} holds. 
Then we have the following commutative diagram: 
\begin{equation}\label{eqn: cd Gr}
\xymatrix{
W_{n+n(i)}F^n/W_{n+n(i-1)}F^n \ar[r]^-{{\res}} & \bigoplus_{F\in \mathcal{C}(i)} H^0(K_{\overline{Y}_{F}})^{{\GF}}  \\ 
{\Mi}/M_{can}^{(i-1)} \ar[u]^{\pi^{\ast}}  \ar[r]^{\Phi^{(i)}_{0}} & 
\bigoplus_{F\in \mathcal{C}(i)}S\!_{cat}(F) \ar[u]_{(\pi_{F}^{\ast})_{F}} 
}
\end{equation}
where all maps are isomorphisms. 
\end{theorem}

\begin{proof}
By Proposition \ref{prop: Shioda monodromy}, we may replace $H^0(K_{{\DFsbar}})$ 
in \eqref{eqn: residue=Siegel} with $H^0(K_{\overline{Y}_{F}})^{{\GF}}$. 
Then we take the direct sum of \eqref{eqn: residue=Siegel} over all $F\in \mathcal{C}(i)$ 
and divide ${\Mi}\stackrel{\simeq}{\to} W_{n+n(i)}F^n$ by 
$M_{can}^{(i-1)}\stackrel{\simeq}{\to} W_{n+n(i-1)}F^n$. 
This gives rise to the commutative diagram 
\begin{equation}\label{eqn: cd Gr pre}
\xymatrix{
W_{n+n(i)}F^n/W_{n+n(i-1)}F^n \ar[r]^-{{\res}}  & \bigoplus_{F\in \mathcal{C}(i)} H^0(K_{\overline{Y}_{F}})^{{\GF}}  \\ 
{\Mi}/M_{can}^{(i-1)} \ar[u]^{\pi^{\ast}}_{\simeq} \ar@{^{(}-_>}[r]^{\Phi^{(i)}_{0}} 
& \bigoplus_{F\in \mathcal{C}(i)}S\!_{cat}(F) \ar@{^{(}-_>}[u]_{(\pi_{F}^{\ast})_{F}}  
}
\end{equation}
By Theorem \ref{thm: GrWFn}, we find that the upper map ${\res}$ is an isomorphism. 
This implies that $(\pi_{F}^{\ast})_{F} \circ \Phi^{(i)}_{0}$ is an isomorphism. 
Since both $(\pi_{F}^{\ast})_{F}$ and $\Phi^{(i)}_{0}$ are injective, 
we see that they are furthermore isomorphisms. 
\end{proof}

\begin{corollary}\label{cor: Siegel surjective}
If the condition \eqref{eqn: n(i) condition} holds, then $\Phi^{(i)}_{0}$ is surjective. 
\end{corollary}

Finally, before proceeding to the proof of Theorem \ref{thm: GrWFn}, 
we shall show that the condition \eqref{eqn: n(i) condition} is often satisfied for $i>1$. 

\begin{lemma}\label{lem: n(i) condition D irreducible}
Suppose that the algebraic group $\mathbb{G}$ is ${\Q}$-simple. 
Then the condition \eqref{eqn: n(i) condition} is satisfied for every $i>1$. 
\end{lemma}

\begin{proof}
We use the notations in \cite{AMRT} \S III.2.5 and \S III.3.5.  
We take normalized adjacent cusps 
$F_{S}\prec F_{S'}$ of corank $i$ and $i-1$ as in \cite{AMRT} Lemma III.3.13 with 
\begin{equation*}
S' = I_{\alpha_{1}} \sqcup \cdots \sqcup I_{\alpha_{i-1}}, \quad 
S = S' \sqcup I_{\alpha_{i}}. 
\end{equation*}
Let $\beta_{\alpha}$ be the ${\Q}$-root associated to $I_{\alpha}$ as in \cite{AMRT} \S III.2.5. 
Possible ${\R}$-roots (hence possible ${\Q}$-roots) of 
$U(F_{S})$ and $U(F_{S'})$ are described in \cite{AMRT} p.143. 
It is assured by \cite{AMRT} Corollary III.2.12 
that the ${\Q}$-roots of $U(F_{S})$ that indeed occur are 
$(\beta_{j}+\beta_{k})/2$ with $1\leq j, k\leq i$,  
and similarly those of $U(F_{S'})$ are $(\beta_{j}+\beta_{k})/2$ with $1\leq j, k\leq i-1$. 
The difference of these two sets is 
\begin{equation*}
\{ (\beta_{j}+\beta_{i})/2 \, | \, 1\leq j \leq i \}. 
\end{equation*}
It follows that 
$\dim (U(F_{S})/U(F_{S'})) \geq  i > 1$. 
This proves our assertion. 
\end{proof}

\subsection{Proof of Theorem \ref{thm: GrWFn}}\label{ssec: proof GrW}

In this subsection we prove Theorem \ref{thm: GrWFn}. 
By \eqref{eqn: GrWF kernel}, ${\Gr}^{W}_{n+n(i)}F^n$ is the kernel of the Gysin map with sign 
\begin{equation*}
H^0(K_{D(n(i))}) \to H^{n-n(i)+1,1}(D(n(i)-1)). 
\end{equation*}
By \eqref{eqn: D(m) toroidal}, the source is written as 
\begin{equation*}
H^0(K_{D(n(i))}) \: = \: 
\bigoplus_{\substack{(F,\sigma) \\ \dim \sigma =n(i)}} H^0(K_{{\DFsbar}}), 
\end{equation*}
where $(F, \sigma)$ runs over ${\G}$-equivalence classes. 
For the same reason as in the proof of Lemma \ref{lem: residue only corank i}, 
only those pairs $(F, \sigma)$ with $F\in \mathcal{C}(i)$ and 
$\sigma$ top-dimensional in $\Sigma_{F}$ contribute here. 
By Lemma \ref{lem: D(m) toroidal}, for such a pair $(F, \sigma)$,  
the connected components of $D(n(i)-1)$ that appear in the target of 
the Gysin map from ${\DFsbar}$ correspond to the codimension $1$ faces of $\sigma$. 
By \eqref{eqn: codim 1 F-cone}, these faces are still $F$-cones. 
(This is a crucial point of the proof.) 
Thus ${\Gr}^{W}_{n+n(i)}F^n$ is the kernel of the Gysin map with sign 
\begin{equation}\label{eqn: signed Gysin toroidal}
\bigoplus_{F\in \mathcal{C}(i)} 
\left( \bigoplus_{\substack{[\sigma]\in \Sigma_{F}/{\GF} \\ \dim \sigma =n(i)}} H^0(K_{{\DFsbar}}) \to 
\bigoplus_{\substack{[\tau]\in \Sigma_{F}/{\GF} \\ \dim \tau =n(i)-1}} H^{n-n(i)+1,1}(\overline{D(F, \tau)}) \right), 
\end{equation}
which is the direct sum of its $F$-components over all $F\in \mathcal{C}(i)$. 
Here note that for each $F$, ${\G}$-equivalence of $F$-cones reduces to ${\GF}$-equivalence. 
As explained in \S \ref{ssec: WFkHk}, the projection 
${\Gr}^{W}_{n+n(i)}F^n \to H^0(K_{{\DFsbar}})$ 
is given by the residue map at ${\DFsbar}$. 

Let $F\in \mathcal{C}(i)$. 
We want to calculate the kernel of the $F$-component of \eqref{eqn: signed Gysin toroidal}. 
The first step is the following normalization. 

\begin{lemma}\label{lem: normalization by YF}
Let $\sigma$ be a top-dimensional cone in $\Sigma_{F}$ and $\tau$ be a codimension $1$ face of $\sigma$. 
The cones $\sigma$ and $\tau$ canonically determine isomorphisms 
$H^0(K_{\overline{Y}_{F}}) \to H^0(K_{{\DFsbar}})$ and 
$H^0(K_{\overline{Y}_F}) \to H^{n-n(i)+1,1}(\overline{D(F, \tau)})$ 
respectively, via which 
the Gysin map $H^0(K_{{\DFsbar}}) \to H^{n-n(i)+1,1}(\overline{D(F, \tau)})$ 
is identified with the identity of $H^0(K_{\overline{Y}_{F}})$. 
\end{lemma}

\begin{proof}
As explained in \S \ref{ssec: GF monodromy}, the cone $\sigma$ determines the isomorphism 
$p_{\sigma} \colon {\DFs} \to Y_{F}$. 
This induces the isomorphism 
$H^0(K_{\overline{Y}_{F}}) \to H^0(K_{{\DFsbar}})$ 
as desired. 
Next, as in \eqref{eqn: P1 bundle},  
the map $p_{\tau}$ makes 
\begin{equation*}
\overline{D(F, \tau)}\cap \pi^{-1}(X_{F}) = D(F, \tau) \sqcup {\DFs} \sqcup D(F, \sigma^{\dag}) 
\end{equation*}
a ${\proj}^1$-bundle over $Y_{F}$ with 
$D(F, \sigma)$ the $0$-section and $D(F, \sigma^{\dag})$ the $\infty$-section, 
where $\sigma^{\dag}\in \Sigma_{F}$ is the top-dimensional cone containing $\tau$ and opposite to $\sigma$. 
We take a birational map 
$\overline{D(F, \tau)} \dashrightarrow \overline{Y}_F\times {\proj}^1$ 
respecting this ${\proj}^1$-bundle structure.  
By the structure of the cohomology of a ${\proj}^1$-bundle, 
the Gysin map $H^0(K_{\overline{Y}_F}) \to H^{n-n(i)+1,1}(\overline{Y}_F\times {\proj}^1)$ for 
$\overline{Y}_F \hookrightarrow {\overline{Y}_F}\times {\proj}^1$ 
is an isomorphism. 
By the birational invariance of the relevant Hodge components, 
we have the commutative diagram 
\begin{equation*}
\xymatrix@C+1pc{
H^0(K_{\overline{Y}_F}) \ar[d]_{p_{\sigma}} \ar[r]^-{\textrm{Gysin}} & 
H^{n-n(i)+1,1}({\overline{Y}_F}\times {\proj}^1) \ar[d]  \\ 
H^0(K_{{\DFsbar}}) \ar[r]^-{\textrm{Gysin}} & H^{n-n(i)+1,1}(\overline{D(F, \tau)})
}
\end{equation*}
where all maps are isomorphisms. 
This gives the isomorphism 
$H^0(K_{\overline{Y}_F}) \to H^{n-n(i)+1,1}(\overline{D(F, \tau)})$ 
as desired. 
Clearly this depends only on $p_{\tau}$.  
\end{proof}

Let us rewrite the isomorphisms in Lemma \ref{lem: normalization by YF} as 
\begin{equation}\label{eqn: HKY sigma}
H^0(K_{\overline{Y}_{F}})\otimes {\C}\langle \sigma \rangle \simeq H^0(K_{{\DFsbar}}), 
\end{equation} 
\begin{equation}\label{eqn: Gysin P1 bundle}
H^0(K_{\overline{Y}_F}) \otimes {\C} \langle \tau \rangle \simeq H^{n-n(i)+1,1}(\overline{D(F, \tau)}), 
\end{equation}
where $\langle \sigma \rangle$ and $\langle \tau \rangle$ are formal symbols 
for keeping track of $\sigma$, $\tau$ respectively. 
Then the Gysin map $H^0(K_{{\DFsbar}}) \to H^{n-n(i)+1,1}(\overline{D(F, \tau)})$ is identified with 
\begin{equation}\label{eqn: Gysin sigma tau} 
{\rm id} \otimes (\langle \sigma \rangle \mapsto \langle \tau \rangle ) \; : \; 
H^0(K_{\overline{Y}_F}) \otimes {\C}\langle \sigma \rangle \to H^0(K_{\overline{Y}_F}) \otimes {\C}\langle \tau \rangle. 
\end{equation}
Note that the labeling isomorphism \eqref{eqn: HKY sigma} 
depends on the cone $\sigma$, rather than its ${\GF}$-equivalence class. 
If we use another equivalent cone $\gamma \sigma$, $\gamma \in {\GF}$, 
the diagram \eqref{CD: gamma act on Dhat} shows that 
the $\sigma$-labeling and the $\gamma \sigma$-labeling 
differ by the $\gamma$-action on $H^0(K_{\overline{Y}_{F}})$. 
The same holds for the labeling \eqref{eqn: Gysin P1 bundle} for the codimension $1$ cone $\tau$. 

Let $K(\Sigma_{F})$ be the simplicial complex associated to the projectivization of $\Sigma_{F}$. 
By Lemma \ref{lem: D(m) toroidal}, the numbering of the irreducible components of ${\Xcpt}-X$ 
determines a numbering of the rays of each $F$-cone, 
or equivalently, the vertices of the associated simplex. 
This determines its orientation. 
(We take arbitrary numberings for the remaining simplices; they will be irrelevant in the following.) 
As explained in \S \ref{ssec: GF monodromy}, 
we can compare the orientation of top-dimensional cones. 
By Lemma \ref{lem: orientation}, whether two such cones have the same orientation depends only on their ${\GF}$-equivalence classes. 


\begin{lemma}\label{lemma: chain map local}
Let $\sigma$ be a top-dimensional cone in $\Sigma_{F}$. 
Via \eqref{eqn: HKY sigma} and \eqref{eqn: Gysin P1 bundle}, 
the component of \eqref{eqn: signed Gysin toroidal} that starts from $H^{0}(K_{{\DFsbar}})$ 
is identified with  
\begin{equation*}
{\rm id} \otimes \partial \; : \: 
H^0(K_{\overline{Y}_{F}}) \otimes 
\left( {\C} \langle \sigma \rangle \to 
\bigoplus_{\tau \prec \sigma} {\C} \langle \tau \rangle \right), 
\end{equation*}
where $\tau$ ranges over all codimension $1$ faces of $\sigma$ and 
$\partial$ is the boundary map in the chain complex of $K(\Sigma_{F})$. 
\end{lemma}

\begin{proof}
Let $\sigma_{1}, \cdots, \sigma_{n(i)}$ be the rays of $\sigma$ (in this order). 
For $1\leq j \leq n(i)$, we denote by $\tau_{j}$ the codimension $1$ face of $\sigma$ spanned by 
$\{ \sigma_{1}, \cdots, \hat\sigma_{j}, \cdots, \sigma_{n(i)} \}$. 
By \eqref{eqn: Gysin sigma tau}, the Gysin map 
$H^0(K_{{\DFsbar}}) \to H^{n-n(i)+1,1}(\overline{D(F, \tau_{j})})$ 
is identified with 
\begin{equation*} 
H^0(K_{\overline{Y}_F}) \otimes {\C} \langle \sigma \rangle \to H^0(K_{\overline{Y}_F}) \otimes {\C} \langle \tau_{j} \rangle, \quad 
{\rm id} \otimes (\langle \sigma \rangle \mapsto \langle \tau_{j} \rangle ). 
\end{equation*}
By our choice of ordering of the vertices, its sign in \eqref{eqn: Gysin sign} is $(-1)^{j-1}$. 
Therefore the map \eqref{eqn: signed Gysin toroidal} on the component $H^0(K_{{\DFsbar}})$ is identified with 
\begin{equation*} 
H^0(K_{\overline{Y}_F}) \otimes {\C}\langle \sigma \rangle \to 
H^0(K_{\overline{Y}_F}) \otimes \bigoplus_{j=1}^{n(i)} {\C}\langle \tau_{j} \rangle, \quad 
{\rm id} \otimes \left( \langle \sigma \rangle \mapsto \sum_{j=1}^{n(i)}(-1)^{j-1}\langle \tau_{j} \rangle \right). 
\end{equation*}
The right side of $\otimes$ is nothing but the boundary map from $\sigma$ 
in the chain complex of $K(\Sigma_{F})$ (see \cite{Hatcher} \S 2.1). 
\end{proof}

The identification in Lemma \ref{lemma: chain map local} depends on the cone $\sigma$, 
rather than its ${\GF}$-equivalence class. 
We pass to a canonical description as follows. 

\begin{lemma}\label{lemma: chain map total}
The $F$-component of \eqref{eqn: signed Gysin toroidal} is canonically identified with 
the ${\GF}$-invariant part of  
\begin{equation}\label{eqn: complex GF invariant}
{\rm id} \otimes \partial \; : \: 
H^0(K_{\overline{Y}_{F}}) \otimes \left( \prod_{\sigma } {\C}\langle \sigma \rangle \to 
\prod_{\tau} {\C} \langle \tau \rangle \right), 
\end{equation}
where $\sigma$ (resp.~$\tau$) ranges over all top-dimensional (resp.~codimension $1$) cones in $\Sigma_{F}$, 
and $\partial$ is the boundary map in the chain complex of $K(\Sigma_{F})$. 
\end{lemma}

Note that although we are considering infinite chains here, 
the boundary map $\partial$ is well-defined because 
the simplices in question are locally finite. 

\begin{proof}
Let $[\sigma]$ be a top-dimensional cone in $\Sigma_{F}/{\GF}$ 
and $\omega$ be an element of $H^0(K_{{\DFsbar}})$. 
For each cone $\sigma\in \Sigma_{F}$ representing $[\sigma]$, 
we denote by $\omega_{\sigma}\in H^0(K_{\overline{Y}_{F}})$ the image of $\omega$ under the map \eqref{eqn: HKY sigma}. 
By taking all such maps simultaneously, we have the canonical isomorphism 
\begin{equation*}\label{eqn: normalization all sigma}
H^0(K_{{\DFsbar}}) \simeq 
\left( \prod_{\sigma} H^0(K_{\overline{Y}_{F}}) \otimes {\C} \langle \sigma \rangle \right)^{{\GF}}, \quad 
\omega \mapsto (\omega_{\sigma})_{\sigma}, 
\end{equation*}
where $\sigma$ ranges over all cones in $\Sigma_{F}$ representing $[\sigma]$. 
Similarly, for a codimension $1$ cone $[\tau]$ in $\Sigma_{F}/{\GF}$, 
we have a canonical isomorphism 
\begin{equation*}\label{eqn: normalization all tau}
H^{n-n(i)+1,1}(\overline{D(F, \tau)}) \simeq 
\left( \prod_{\tau} H^0(K_{\overline{Y}_{F}}) \otimes {\C} \langle \tau \rangle \right)^{{\GF}} 
\end{equation*}
by taking \eqref{eqn: Gysin P1 bundle} for all cones $\tau$ in $\Sigma_{F}$ representing $[\tau]$. 
Then our assertion follows from Lemma \ref{lemma: chain map local}. 
\end{proof}

In what follows, we use the notation ${\rm id}\otimes \partial$ for the map \eqref{eqn: complex GF invariant}. 
By Lemma \ref{lemma: chain map total}, the $F$-component of 
${\rm Gr}^{W}_{n+n(i)}F^n$ is identified with $\ker({\rm id}\otimes \partial)^{{\GF}}$. 
We express an element of 
$\prod_{\sigma } H^0(K_{\overline{Y}_{F}}) \otimes {\C}\langle \sigma \rangle$ 
as $(\omega_{\sigma})_{\sigma}$ 
where $\sigma$ ranges over all top-dimensional cones in $\Sigma_{F}$ and 
$\omega_{\sigma}\in H^{0}(K_{\overline{Y}_{F}})$ for each $\sigma$. 
The final step of the proof of Theorem \ref{thm: GrWFn} is the following. 

\begin{lemma}\label{lem: GF invariant part}
Let $\sigma_{0}\in \Sigma_{F}$ be an arbitrary top-dimensional cone. 
Projecting $(\omega_{\sigma})_{\sigma}$ to $\omega_{\sigma_{0}}$ defines an isomorphism 
\begin{equation*}
\ker({\rm id}\otimes \partial)^{{\GF}} \simeq H^{0}(K_{\overline{Y}_{F}})^{{\GF}}. 
\end{equation*}
\end{lemma}

\begin{proof}
For a top-dimensional cone $\sigma\in \Sigma_{F}$, 
we define ${\rm sgn}(\sigma)\in \{ \pm 1 \}$ according to 
whether $\sigma$ has the same orientation as $\sigma_{0}$ or not. 
An element $(\omega_{\sigma})_{\sigma}$ of 
$\prod_{\sigma}H^0(K_{\overline{Y}_{F}})\otimes {\C}\langle \sigma \rangle$ 
is annihilated by ${\rm id}\otimes \partial$ if and only if 
we have $\omega_{\sigma_{+}}=\pm \omega_{\sigma_{-}}$ for any two adjacent cones $\sigma_{+}, \sigma_{-}$, 
where the sign is determined according to whether 
$\sigma_{+}$ and $\sigma_{-}$ have the same orientation. 
Therefore we have the isomorphism 
\begin{equation}\label{eqn: Ker id partial}
H^0(K_{\overline{Y}_{F}}) \to \ker({\rm id}\otimes \partial), \quad 
\omega \mapsto ({\rm sgn}(\sigma) \omega)_{\sigma}. 
\end{equation}
The group ${\GF}$ acts on $\ker({\rm id}\otimes \partial)$ by 
\begin{equation*}
\gamma : ({\rm sgn}(\sigma) \omega)_{\sigma} \mapsto 
({\rm sgn}(\gamma^{-1}\sigma) \gamma^{-1}\omega)_{\sigma} = 
({\rm sgn}(\sigma) \gamma^{-1}\omega)_{\sigma}, 
\end{equation*}
where the equality 
${\rm sgn}(\gamma^{-1}\sigma)={\rm sgn}(\sigma)$ 
follows from Lemma \ref{lem: orientation}. 
Hence, via the isomorphism \eqref{eqn: Ker id partial}, 
the ${\GF}$-action on $\ker({\rm id}\otimes \partial)$ 
is identified with the natural ${\GF}$-action on $H^0(K_{\overline{Y}_{F}})$. 
This proves our assertion. 
\end{proof}

This shows that the composition 
\begin{equation*}
{\rm Gr}^{W}_{n+n(i)}F^{n} \stackrel{{\res}_{F,\sigma}}{\longrightarrow} 
H^{0}(K_{\overline{D(F, \sigma)}}) \stackrel{p_{\sigma}}{\longrightarrow} 
H^0(K_{\overline{Y}_{F}}) 
\end{equation*}
defines an isomorphism from the $F$-component of ${\rm Gr}^{W}_{n+n(i)}F^{n}$ 
to $H^0(K_{\overline{Y}_{F}})^{{\GF}}$.    
The proof of Theorem \ref{thm: GrWFn} is now completed.

\begin{remark}
The isomorphism \eqref{eqn: Ker id partial} can be regarded as the description of 
the top-dimensional (Borel-Moore) homology of $K(\Sigma_{F})$ with coefficients in $H^{0}(K_{\overline{Y}_{F}})$ 
as the tensor product of $H^{0}(K_{\overline{Y}_{F}})$ with the fundamental class of $K(\Sigma_{F})$. 
Taking the ${\GF}$-invariant part gives the description of the corresponding homology with local coefficients of 
the quotient complex $K(\Sigma_{F})/{\GF}$ 
as the tensor product of $H^{0}(K_{\overline{Y}_{F}})^{{\GF}}$ with the fundamental class of $K(\Sigma_{F})/{\GF}$. 
This is a geometric interpretation behind the above argument. 
\end{remark}

\subsection{On the case $n(1)=1$}\label{ssec: n(1)=1}

In this subsection we look at the case $i=1$ with $n(1)=1$. 
This is a typical case where \eqref{eqn: n(i) condition} does not hold; 
when $\mathbb{G}$ is ${\Q}$-simple, this is the only exceptional case by Lemma \ref{lem: n(i) condition D irreducible}.  
In this case, $\Sigma_{F}$ consists of only one ray $\sigma$ for $F\in \mathcal{C}(1)$. 
So the toroidal compactification over $F$ is unique, and  
we can identify the unique boundary divisor ${\DFs}$ over $X_{F}$ with $Y_{F}$. 
The group ${\GF}$ is trivial. 
Then we have the following description of ${\Gr}^{W}_{n+1}F^n$. 

\begin{proposition}\label{prop: n(1)=1}
When $n(1)=1$, we have the exact sequence 
\begin{equation*}\label{eqn: exact seq n(1)=1}
0 \to {\Gr}^{W}_{n+1}F^n \stackrel{{\res}}{\to} \bigoplus_{F\in \mathcal{C}(1)} H^0(K_{\overline{Y}_{F}}) \to 
H^{n,1}({\Xcpt}) \to F^nW_{n+1}H^{n+1}(X) \to 0, 
\end{equation*}
where the middle map is the Gysin map.  
\end{proposition}

\begin{proof}
The first part of the proof of Theorem \ref{thm: GrWFn} is still valid and shows 
\begin{equation*}
{\Gr}^{W}_{n+1}F^{n}  \: \simeq \: 
\ker \left( \bigoplus_{F\in \mathcal{C}(1)} H^0(K_{\overline{Y}_{F}}) \stackrel{d_{1}}{\longrightarrow} H^{n,1}({\Xcpt}) \right). 
\end{equation*}
Here $-d_{1}$ is the Gysin map. (The sign $(-1)^{j-1}$ in \eqref{eqn: Gysin sign} is $1$.) 
The cokernel of $d_1$ is  
\begin{equation*}
F^nE_{2}^{0,n+1} = F^n{\Gr}^{W}_{n+1}H^{n+1}(X) = F^nW_{n+1}H^{n+1}(X). 
\end{equation*}
This proves our assertion. 
\end{proof}

From this we obtain the following description of the cokernel of $\Phi_{0}^{(1)}$. 

\begin{proposition}\label{prop: cokernel Siegel n(1)=1}
When $n(1)=1$, we have an injective map 
\begin{equation}\label{eqn: Phi1 Omegan-1}
\coker(\Phi_{0}^{(1)}) \hookrightarrow \ker (H^{n,1}({\Xcpt}) \to F^nW_{n+1}H^{n+1}(X)). 
\end{equation}
\end{proposition}

\begin{proof}
The commutative diagram \eqref{eqn: cd Gr pre} is still valid in this case.  
This shows that the map $(\pi_{F}^{\ast})_{F\in \mathcal{C}(1)}$  
defines an injective map 
\begin{equation*}
\coker(\Phi_{0}^{(1)}) \hookrightarrow 
\coker \left( {\rm Gr}^{W}_{n+1}F^n \stackrel{{\res}}{\to} \bigoplus_{F\in \mathcal{C}(1)} H^0(K_{\overline{Y}_{F}}) \right).   
\end{equation*}
By Proposition \ref{prop: n(1)=1}, the right hand side is isomorphic to the kernel of 
$H^{n,1}({\Xcpt}) \to F^nW_{n+1}H^{n+1}(X)$ by the Gysin map. 
\end{proof} 

Proposition \ref{prop: cokernel Siegel n(1)=1} shows that the space 
\begin{equation*}
H^{n,1}({\Xcpt}) \simeq H^0({\Xcpt}, \Omega^{n-1})\simeq H^0(X, \Omega^{n-1}) 
\end{equation*}
provides an obstruction space for $\Phi_{0}^{(1)}$ to be surjective. 
Here the second isomorphism follows from the extension theorem of Pommerening \cite{Po}.

\subsection{More on Siegel operators}\label{ssec: more Siegel}

In this subsection we reformulate surjectivity of the total Siegel operator for $M_{can}$, rather than ${\Mi}$. 
We assume that the algebraic group $\mathbb{G}$ is ${\Q}$-simple. 
We denote by 
\begin{equation*}
\Phi^{(i)} = (\Phi_{F})_{F} : M_{can} \to \bigoplus_{F\in \mathcal{C}(i)}M_{cat}(F) 
\end{equation*}
the direct sum of the Siegel operators from $M_{can}$ to all cusps of corank $i$. 
We define a subspace 
$V_{i}$ of $\bigoplus_{F\in \mathcal{C}(i)}M_{cat}(F)$ 
as follows. 
First we put 
\begin{equation*}
M_{cat}(F)^{+} = H^0(\overline{X}_{F}, {\LL}|_{\overline{X}_{F}}) 
\end{equation*}
for each cusp $F$. 
This is a subspace of $M_{cat}(F)$ consisting of 
modular forms on $F$ whose values at two boundary points of $X_{F}^{bb}$ 
identified at $\overline{X}_{F}$ coincide. 
In particular, $S\!_{cat}(F)$ is contained in $M_{cat}(F)^{+}$. 
Clearly the image of the Siegel operator $\Phi_{F}$ at $F$ is contained in $M_{cat}(F)^{+}$. 
If $F'\prec F$, then $\overline{X}_{F'}\subset \overline{X}_{F}$, 
so the restriction map 
\begin{equation*}
M_{cat}(F)^{+} \to M_{cat}(F')^{+}, \quad f_{F}\mapsto f_{F}|_{F'}, 
\end{equation*}
is well-defined. 
Now we define $V_{i}$ as a subspace of $\bigoplus_{F\in \mathcal{C}(i)}M_{cat}(F)^{+}$ by    
\begin{equation*} 
V_i \: = \: \{ \, (f_{F})_{F\in \mathcal{C}(i)} \; | \; (f_{F_{1}})|_{F'} = (f_{F_{2}})|_{F'} \; \textrm{if} \; F'\prec F_1, F'\prec F_2 \, \}. 
\end{equation*} 
Thus $V_{i}$ consists of tuples of modular forms on the cusps of corank $i$ 
whose restrictions to smaller cusps are compatible. 
 
\begin{proposition}\label{thm: Siegel surjective}
Assume that the algebraic group $\mathbb{G}$ is ${\Q}$-simple and 
let $1< i \leq r$ or $i=1$ with $n(1)>1$. 
Then the image of $\Phi^{(i)}$ is $V_i$. 
\end{proposition}

\begin{proof}
The inclusion $\Phi^{(i)}(M_{can})\subset V_{i}$ is obvious because 
$\Phi_{F}(f)|_{F'}=\Phi_{F'}(f)$ for $f\in M_{can}$ and $F'\prec F$. 
We shall show that $\Phi^{(i)}\colon M_{can}\to V_{i}$ is surjective. 
We use induction on $i$, starting from $i=r$ and decreasing. 
When $i=r$, cusps $F$ of corank $r$ are minimal cusps, 
so $X_{F}$ are closed and we have 
\begin{equation*}
V_r = \bigoplus_{F\in \mathcal{C}(r)} M_{cat}(F) = \bigoplus_{F\in \mathcal{C}(r)} S\!_{cat}(F). 
\end{equation*}
Then $\Phi^{(r)}=\Phi^{(r)}_{0}$ is surjective by 
Corollary \ref{cor: Siegel surjective} and Lemma \ref{lem: n(i) condition D irreducible}. 

Assuming that the assertion is proved in corank $i+1$, we proceed to corank $i$. 
By the ${\Q}$-simplicity of $\mathbb{G}$, every cusp of corank $i+1$ is in the closure of a cusp of corank $i$ 
(\cite{BB} Theorem 3.8). 
Therefore we can consider the intermediate total Siegel operator 
\begin{equation*}
\Phi^{(i+1)}_{(i)} : V_{i} \to V_{i+1}, \quad 
(f_{F})_{F\in \mathcal{C}(i)}\mapsto (f_{F}|_{F'})_{F'\in \mathcal{C}(i+1)}, 
\end{equation*}
which restricts modular forms on cusps of corank $i$ to those of corank $i+1$. 
This map is well-defined by the definition of $V_{i}$. 
The composition 
\begin{equation*}
\Phi^{(i+1)}_{(i)} \circ \Phi^{(i)} : M_{can} \to V_{i} \to V_{i+1} 
\end{equation*}
coincides with $\Phi^{(i+1)}$. 
By the assumption of induction, this composition is surjective. 
Moreover, we have 
\begin{equation*}
\ker (\Phi^{(i+1)}_{(i)})=\bigoplus_{F\in \mathcal{C}(i)} S\!_{cat}(F), 
\end{equation*}
because the boundary $\overline{X}_{F}-X_{F}$ of $X_{F}$ for $F\in \mathcal{C}(i)$ 
is the union of $\overline{X}_{F'}$ for cusps $F'\prec F$ of corank $i+1$ 
by the ${\Q}$-simplicity of $\mathbb{G}$ (\cite{BB} Corollary 3.9). 
By our assumption on $i$, 
this space is contained in ${\rm Im}(\Phi^{(i)})$ 
by Corollary \ref{cor: Siegel surjective} and Lemma \ref{lem: n(i) condition D irreducible}. 
Now the surjectivity of $\Phi^{(i+1)}_{(i)} \circ \Phi^{(i)}$ and 
the inclusion $\ker (\Phi^{(i+1)}_{(i)}) \subset {\rm Im}(\Phi^{(i)})$ 
imply the surjectivity of $\Phi^{(i)}$. 
\end{proof}

\section{Non-neat case}\label{sec: non-neat}

In \S \ref{sec: VT I} -- \S \ref{sec: GrW}, 
we assumed that the arithmetic group ${\G}$ is neat. 
In this section we show that some of the main results of \S \ref{sec: VT I} -- \S \ref{sec: GrW} 
are still valid for non-neat groups. 
Basically assertions whose statements do not involve residue maps 
can be extended to the non-neat case. 

Let ${\D}$ be a Hermitian symmetric domain and ${\G}$ be an arithmetic group acting on ${\D}$, 
not necessarily neat. 
Let $X={\D}/{\G}$. 
This is a quasi-projective variety, no longer smooth in general, 
but its cohomology still has a canonical mixed Hodge structure (\cite{DeIII}). 
We keep the assumption that the Baily-Borel boundary has codimension $\geq 2$. 

We first consider the assertions concerning the Hodge numbers of $H^k(X)$. 
We choose a neat normal subgroup ${\G}'\lhd {\G}$ of finite index and write 
$X'={\D}/{\G}'$ and $G={\G}/{\G}'$. 
Then $X$ is the quotient of the smooth variety $X'$ by the finite group $G$. 
Let $p\colon X'\to X$ be the projection. 
Let $(W_{\bullet}, F^{\bullet})$ be the mixed Hodge structure on $H^k(X', {\Q})$. 
We have the isomorphism (see \cite{Br} \S  III.2) 
\begin{equation*}
p^{\ast} : H^k(X, {\Q}) \stackrel{\simeq}{\to} H^k(X', {\Q})^{G} \subset H^k(X', {\Q}). 
\end{equation*}
This is an isomorphism of mixed Hodge structures, 
where the weight and the Hodge filtrations on $H^k(X', {\Q})^{G}$ are 
the $G$-invariant parts $(W_{\bullet}^{G}, (F^{\bullet})^{G})$ of $(W_{\bullet}, F^{\bullet})$ 
by the strictness of morphisms of mixed Hodge structures. 
By the complete reducibility of the $G$-representation on $H^k(X')$, 
we have a canonical isomorphism 
\begin{equation*}
{\Gr}^{W^{G}}_{k+m}(H^k(X')^{G}) \simeq ({\Gr}^{W}_{k+m}H^k(X'))^{G}. 
\end{equation*}
Hence $p^{\ast}$ induces an isomorphism of pure Hodge structures 
\begin{equation}\label{eqn: G-invariant Gr}
p^{\ast} : {\Gr}^{W}_{k+m}H^k(X) \stackrel{\simeq}{\to} ({\Gr}^{W}_{k+m}H^k(X'))^{G}\subset  {\Gr}^{W}_{k+m}H^k(X'). 
\end{equation}
In particular, the range of $(p, q)$ with nonzero Hodge number satisfies 
the same constraint \eqref{eqn: (p,q) range} as the smooth case. 

\begin{proposition}
The statements of Lemma \ref{lem: pure BB}, Theorem \ref{thm: VT I}, 
Proposition \ref{thm: VT II}, Corollary \ref{cor: VT II} and Corollary \ref{cor: 0-dim cusp}
hold for non-neat groups ${\G}$ as well. 
\end{proposition}

\begin{proof}
Except for Corollary \ref{cor: 0-dim cusp}, 
these are vanishing assertions for some Hodge numbers. 
By the embedding \eqref{eqn: G-invariant Gr}, 
the vanishing assertions for ${\Gr}^{W}_{k+m}H^k(X')$ 
imply the corresponding ones for ${\Gr}^{W}_{k+m}H^k(X)$. 

As for Corollary \ref{cor: 0-dim cusp}, we make use of the Siegel operators. 
The isomorphism $M_{can}({\G}')\to F^nH^n(X')$ is $G$-equivariant 
(independent of toroidal compactifications). 
Hence it induces an isomorphism 
\begin{equation*}
(M_{can}({\G}')/M_{can}^{(r-1)}({\G}'))^{G} \to 
({\Gr}^{W}_{2n}F^nH^n(X'))^{G} = ({\Gr}^{W}_{2n}H^n(X'))^{G} \simeq {\Gr}^{W}_{2n}H^n(X). 
\end{equation*}
The isomorphism 
\begin{equation*}
M_{can}({\G}')/M_{can}^{(r-1)}({\G}') \stackrel{\simeq}{\to} 
\bigoplus_{F\in \mathcal{C}(r)'} S\!_{cat}(F, {\G}'(F)) 
\end{equation*}
given by the Siegel operators is also $G$-equivariant, 
where $\mathcal{C}(r)'$ is the set of $0$-dimensional cusps of $X'$ and 
each $S\!_{cat}(F, {\G}'(F))$ is $1$-dimensional. 
Since $G$ acts on the right hand side by the permutation on $\mathcal{C}(r)'$, 
taking the $G$-invariant part shows that $\dim {\Gr}^{W}_{2n}H^n(X)$ is equal to $|\mathcal{C}(r)'/G|$, 
which is the number of $0$-dimensional cusps of $X$. 
\end{proof}

Next we consider the assertions about the Siegel operators. 
The space $M_{can}({\G})$ of ${\G}$-modular forms of canonical weight is defined  
as $H^{0}({\D}, K_{{\D}})^{{\G}}$. 
This is the $G$-invariant part of $M_{can}({\G}')$. 
For a cusp $F$ of ${\D}$, we write $G_{F} = {\GFZ}/{\G}'(F)$. 
This is the stabilizer of the $F$-locus $X'_{F}\subset (X')^{bb}$ in $G$. 
We set 
\begin{equation*}
M_{cat}(F, {\GFZ}) = M_{cat}(F, {\G}'(F))^{G_F}, \quad 
S\!_{cat}(F, {\GFZ}) = S\!_{cat}(F, {\G}'(F))^{G_F}, 
\end{equation*}
which do not depend on the choice of ${\G}'$. 
The Siegel operator 
\begin{equation*}
\Phi_{F}\colon M_{can}({\G}) \to M_{cat}(F, {\GFZ}) 
\end{equation*} 
is defined also for ${\G}$ as it is actually defined at the level of ${\D}$ (\cite{BB}). 
This is the restriction of the Siegel operator for ${\G}'$ to the $G$-invariant part. 
The total Siegel operator for ${\G}$ 
\begin{equation*}
M_{can}({\G}) \to \bigoplus_{F\in \mathcal{C}(i)} M_{cat}(F, {\GFZ}) 
\end{equation*}
is identified with the $G$-invariant part of the total Siegel operator for ${\G}'$. 
Hence the filtration $({\Mi}({\G}))_{i}$ can be defined similarly, 
and is the $G$-invariant part of the filtration $(M_{can}^{(i)}({\G}'))_{i}$ on $M_{can}({\G}')$. 
 
\begin{proposition}\label{prop: Siegel non-neat}
The following holds for non-neat groups ${\G}$ as well. 

(1) 
We have a natural isomorphism $F^nH^n(X)\simeq M_{can}({\G})$. 
This maps the weight filtration $(W_{n+n(i)}F^nH^n(X))_{i}$ 
to the filtration $({\Mi}({\G}))_{i}$. 

(2) Let $1\leq i \leq r$. 
If \eqref{eqn: n(i) condition} holds, 
the cuspidal total Siegel operator  
\begin{equation*}
{\Mi}({\G}) \to \bigoplus_{F\in \mathcal{C}(i)} S\!_{cat}(F, {\GFZ}) 
\end{equation*}
is surjective. 
\end{proposition}

\begin{proof}
(1) 
Since the isomorphism $F^nH^n(X')\simeq M_{can}({\G}')$ for ${\G}'$ is $G$-equivariant, 
it induces an isomorphism $F^nH^n(X)\simeq M_{can}({\G})$ between the $G$-invariant parts. 
The weight filtration on $F^nH^n(X)$ is the $G$-invariant part of the weight filtration on $F^nH^n(X')$,  
while the filtration $({\Mi}({\G}))_{i}$ on $M_{can}({\G})$ is the $G$-invariant part of 
the filtration $(M_{can}^{(i)}({\G}'))_{i}$ on $M_{can}({\G}')$. 
Therefore the two filtrations agree. 

(2) The cuspidal total Siegel operator for ${\G}$ is the $G$-invariant part of 
the cuspidal total Siegel operator for ${\G}'$. 
The map between the $G$-invariant parts of a surjective $G$-homomorphism remains surjective 
by the complete reducibility of $G$-representations and Schur's lemma. 
\end{proof} 

The non-neat version of Proposition \ref{thm: Siegel surjective} also holds 
if we simply set $V_i({\G})=V_i({\G}')^{G}$, 
which can be considered as the space of tuples of modular forms on the ${\G}$-cusps of corank $i$ 
whose restrictions to smaller cusps are compatible.

\section{Examples}\label{sec: example}

In this section we explain the results of \S \ref{sec: VT I} -- \S \ref{sec: non-neat} more explicitly 
for some classical examples of ${\D}$ and compare them with known results.

\subsection{Hilbert modular varieties}\label{ssec: Hilbert}

The study of mixed Hodge structures of Hilbert modular varieties 
was initiated by Oda \cite{Od} in the case $k=n$, 
and completed by Ziegler (\cite{Fr} \S III.7) for general $k$. 
See \cite{Fr} Theorem III.7.9 for the final result. 
In this case our results add nothing. 

What played a key role in \cite{Od}, \cite{Fr} 
is Harder's theory of Eisenstein cohomology, which is fully developed for Hilbert modular varieties 
(see \cite{Har1}, \cite{Har2} etc). 
For general ${\D}$, the relation between Eisenstein cohomology 
and mixed Hodge structure seems to remain rather mysterious (cf.~\cite{OS1}, \cite{HZII}, \cite{Na}).

\subsection{Siegel modular varieties}\label{ssec: Siegel}
 
Let $X={\D}/{\G}$ be a Siegel modular variety 
where ${\D}$ is the Siegel upper half space of genus $g>1$ and ${\G}<{\rm Sp}(2g, {\Q})$. 
We have $\dim X=g(g+1)/2$, 
and the Baily-Borel boundary has codimension $g$. 
Thus $H^k(X)$ is pure when $k<g$. 
The ${\Q}$-rank of ${\rm Sp}(2g, {\Q})$ is $g$. 
For $1\leq g' \leq g$ we have 
$n(g')=g'(g'+1)/2$. 
Indeed, $U(F)$ is isomorphic to the space of real symmetric matrices of size $g'$. 

Miyazaki \cite{Mi} proved that $\dim {\Gr}_{2n}^{W} H^n(X)$ 
is equal to the number of $0$-dimensional cusps and $\dim {\Gr}_{2n-1}^{W} H^n(X)=0$ 
for principal congruence subgroups of ${\rm Sp}(2g, {\Z})$ of level $\geq 3$. 
Corollary \ref{cor: 0-dim cusp} extends his first result, 
and Theorem \ref{thm: VT I} extends his second result. 
Mixed Hodge structures of Siegel modular $3$-folds were first studied by Oda-Schwermer \cite{OS1}. 
The Hodge numbers of some special Siegel modular $3$-folds were completely determined by 
Hoffman-Weintraub \cite{HW1}, \cite{HW2}. 

The canonical weight for Siegel modular forms is $g+1$. 
Corollary \ref{cor: Siegel surjective} and Proposition \ref{thm: Siegel surjective} say that in weight $g+1$, 
the total Siegel operators of corank $>1$ are surjective. 
For $0$-dimensional cusps, this was essentially proved by Shimura \cite{Shimura}. 
When ${\G}={\rm Sp}(2g, {\Z})$, Siegel operators of arbitrary corank were 
studied by Weissauer \cite{We} thoroughly. 
In this case, our result is covered by his result. 
Both Shimura and Weissauer used the Hecke summation of Eisenstein series. 

The phenomenon that the defect of surjectivity of Siegel operators of corank $1$ is related to $H^0(X, \Omega^{n-1})$ 
was first observed by Weissauer for ${\G}={\rm Sp}(2g, {\Z})$ 
in the form of vector-valued modular forms (\cite{We} Satz 13). 
Proposition \ref{prop: cokernel Siegel n(1)=1} shows that this is a general phenomenon, 
and offers a geometric explanation. 
An explicit non-surjective example is given in \cite{SM} Theorem 5.

\subsection{Orthogonal modular varieties}\label{ssec: orthogonal}

Let $X={\D}/{\G}$ be a modular variety associated to 
an integral quadratic form of signature $(2, n)$ with $n\geq 3$, say of Witt index $2$, 
where ${\G}$ is an arithmetic subgroup of the orthogonal group. 
In this case $\dim X= n$, and $X^{bb}$ has only $0$-dimensional and $1$-dimensional cusps. 
Thus the ${\Q}$-rank is $2$, and $H^k(X)$ is pure when $k<n-1$. 
We have $n(1)=1$ and $n(2)=n$, so the Hodge stairs take a simple shape. 
By Corollary \ref{cor: 0-dim cusp}, $\dim {\Gr}^{W}_{2n}H^{n}(X)$ is equal to the number of $0$-dimensional cusps. 

The canonical weight is $n$, 
where weight $1$ corresponds to the tautological line bundle on the compact dual of ${\D}$, 
the isotropic quadric. 
(We have to take the twist by the determinant character when ${\G}$ has torsion.) 
Corollary \ref{cor: Siegel surjective} says that in weight $n$, 
the total Siegel operator to $0$-dimensional cusps is surjective. 
This seems to be new at least as an explicit statement. 
Indyk \cite{In} studied analytic continuation of Eisenstein series from the viewpoint of $L$-functions. 
Probably surjectivity of the Siegel operator could also be derived by this approach.

\subsection{Unitary modular varieties}\label{ssec: unitary}

We consider the unitary group ${\rm U}(p, q)$ 
of a Hermitian form of signature $(p, q)$ with $p\leq q$ over an imaginary quadratic field, 
say of Witt index $p$ (maximal). 
Let $X={\D}/{\G}$ be an associated modular variety, 
where ${\G}$ is an arithmetic subgroup of ${\rm U}(p, q)$. 
When $p=1$, ${\D}$ is a complex ball. 
When $p=q$, ${\D}$ is a so-called Hermitian upper half space. 
We have $\dim X=pq$, 
and the Baily-Borel boundary has codimension $p+q-1$. 
Hence $H^k(X)$ is pure when $k<p+q-1$. 
The ${\Q}$-rank of ${\rm U}(p, q)$ is $p$ (the Witt index). 
For $1\leq i\leq p$ we have $n(i)=i^2$. 
Indeed, $U(F)$ is isomorphic to the space of Hermitian matrices of size $i$. 
Therefore, when $p<q$, we have ${\Gr}^{W}_{n+m}H^n(X)=0$ in the range $p^2 < m \leq n$. 
On the other hand, when $p=q$, 
Corollary \ref{cor: 0-dim cusp} says that 
$\dim {\Gr}^{W}_{2n}H^n(X)$ is equal to the number of $0$-dimensional cusps. 

The canonical weight is $p+q$, where weight $1$ corresponds to 
the dual Pl\"ucker line bundle on the Grassmannian ${\rm G}(p, p+q)$, the compact dual of ${\D}$. 
Corollary \ref{cor: Siegel surjective} and Proposition \ref{thm: Siegel surjective} say that in weight $p+q$, 
the total Siegel operators of corank $>1$ are surjective. 
In the case of $0$-dimensional cusps of the Hermitian upper half spaces ($p=q$), 
this was essentially proved by Shimura \cite{Shimura} using the Hecke summation of Eisenstein series.

\appendix 

\section{Toroidal compactifications with SNC boundary divisor}\label{sec: SNC}

Let ${\D}$ be a Hermitian symmetric domain and ${\G}$ be a neat arithmetic group acting on ${\D}$. 
Let $X={\D}/{\G}$. 
We can take a smooth projective toroidal compactification ${\Xcpt}$ of $X$ (\cite{AMRT}), 
where the smoothness is assured by requiring that every cone $\sigma \in \Sigma_{F}$ 
is generated by a part of a ${\Z}$-basis of ${\UFZ}$. 
In this situation, the boundary divisor $D$ is normal crossing. 
However, normal crossing is weaker than \textit{simple} normal crossing (SNC), 
the latter requiring furthermore that every irreducible component of $D$ is smooth, 
or in other words, has no self-intersection. 
In order to compute the mixed Hodge structure of $X$ with this compactification, we need $D$ to be SNC. 

In this appendix we provide a proof of the following well-known folklore result, 
which has been used implicitly in the literature.  

\begin{proposition}\label{prop: SNC exist}
There exists a smooth projective toroidal compactification of $X$ with SNC boundary divisor. 
\end{proposition}

The history of this folklore result seems to have some kind of self-intersections. 
In \cite{FC} (p.98 (ii) and Remark IV.5.8 (a)), Faltings and Chai 
stated a condition on the fans $\Sigma$ for the boundary to be SNC, 
and suggested its existence, both without proof. 
In \cite{YZ} \S 2.4, Yau and Zhang gave an argument that the Faltings-Chai condition indeed implies SNC. 
However, it was pointed out later by Lan (from another viewpoint) that   
the FC condition has to be modified: see Condition 6.2.5.25 in the revised version of \cite{La}. 
The relevance of this correction to the SNC matter was noticed by R\"ossler (\cite{Ro} Lemma 3.1.10). 
R\"ossler suggested that the argument of Yau-Zhang still works with the Faltings-Chai-Lan condition, 
but without detail. 
Moreover, it seems that existence is not proved in both \cite{YZ} and \cite{Ro}. 
This is suggested by Lan (\cite{La}, the paragraph just before Definition 6.2.5.28), but without detail.   

In fact, around the same time as Faltings-Chai, 
Pink gave an SNC condition in the setting of mixed Shimura varieties with a proof 
(\cite{Pi} Condition 7.12 and Corollary 7.17 (b)). 
However, Pink proved the existence only for a subgroup of the given group ${\G}$ (\cite{Pi} Proposition 7.13), 
which is not sufficient for our purpose. 

All in all, the relevant pieces are scattered or only implicit in the literature. 
Therefore, in order to secure the validity of this folklore result, 
we decided to provide a coherent and self-contained proof in this appendix. 
In \S \ref{ssec: SNC condition toroidal}, we give a simple SNC condition in the spirit of Pink. 
In \S \ref{ssec: SNC existence}, we prove existence of $\Sigma$ satisfying this condition. 
We will not pursue the relation with the FCL condition, 
as our purpose is to secure the existence. 

\subsection{A criterion}\label{ssec: SNC condition toroidal}

Let ${\Xcpt}$ be a smooth toroidal compactification of $X$. 
If $F, F'$ are cusps of ${\D}$ with $F\prec F'$, 
we regard $\Sigma_{F'}\subset \Sigma_{F}$ via the inclusion $U(F')\subset U(F)$. 
In this way, cones in $\Sigma$ are glued. 
Then ${\G}$ acts on $\Sigma$ by the ${\G}$-admissibility of $\Sigma$. 
We consider the following condition on $\Sigma$: 
 
\begin{condition}\label{separate condition}
For any cone $\sigma\in \Sigma$, 
two different rays of $\sigma$ cannot be ${\G}$-equivalent. 
\end{condition}

Our collection of fans $\Sigma$ here is smooth, 
but note that this condition makes sense even when $\Sigma$ is not smooth. 
Note also that ${\G}$-equivalence of two rays in $\Sigma_{F}$ is stronger than ${\G}_{F}$-equivalence 
when they are in the boundary of $C(F)$. 
(We are prohibiting the former.) 
To prohibit only ${\G}_{F}$-equivalence is not sufficient: 
this tells us only that the restriction of the boundary divisor $D$ 
to a neighborhood $U$ of the $F$-locus $\pi^{-1}(X_{F})$ is SNC,  
but does not rule out the possibility that two irreducible components of $D\cap U$ are 
analytic open subsets of the same irreducible component of $D$. 

Condition \ref{separate condition} corresponds more or less to Pink's Condition 7.12 in \cite{Pi}, 
except for the differences that 
Pink works in the language of (mixed) Shimura varieties and considers not only rays but also faces of $\sigma$. 

In Lemma \ref{lem: D(m) toroidal}, we showed that 
Condition \ref{separate condition} is necessary for $D$ to be SNC. 
We prove that it is also sufficient.  

\begin{proposition}\label{prop: SNC condition}
Suppose that Condition \ref{separate condition} holds for the smooth $\Sigma$.  
Then the boundary divisor $D$ of ${\Xcpt}$ is simple normal crossing. 
\end{proposition}

\begin{proof}
Since we already know that $D$ is normal crossing, 
it is sufficient to check that its irreducible components have no self-intersection. 
Let $x$ be a boundary point of ${\Xcpt}$. 
Let ${\DFs}$ be the stratum of ${\Xcpt}$ containing $x$. 
We choose a point $\tilde{x} \in {\SFs}\subset {\XFcpt}$ which is mapped to $x$ by the projection 
$p\colon {\XFcpt}\to {\Xcpt}$. 
Since ${\G}$ is neat, $p$ is a local isomorphism in a small open neighborhood $U$ of $\tilde{x}$. 
If $\tau_{1}, \cdots, \tau_{m}$ are the rays of $\sigma$, 
the intersection of the boundary divisor of ${\XFcpt}$ with $U$ is 
\begin{equation*}
U \cap (\overline{S\!(F, \tau_{1})} + \cdots + \overline{S\!(F, \tau_{m})}), 
\end{equation*}
where $\overline{S\!(F, \tau_{1})}, \cdots, \overline{S\!(F, \tau_{m})}$ are distinct. 
Therefore we have 
\begin{equation*}
p(U) \cap D = p(U \cap \overline{S\!(F, \tau_{1})}) + \cdots + p(U \cap \overline{S\!(F, \tau_{m})}), 
\end{equation*}
where the $m$ components in the right hand side are distinct. 
If we write $D_{i}=\overline{D(F_{i}, \tau_{i})}$, 
then $p(U \cap \overline{S\!(F, \tau_{i})})$ is an analytic open subset of $D_{i}$. 
By Condition \ref{separate condition}, the $m$ divisors $D_{1}, \cdots, D_{m}$ are globally distinct. 
This means that no irreducible component of $D$ has self-intersection in a neighborhood of $x$. 
Since $x$ is an arbitrary boundary point, this proves our assertion. 
\end{proof}

\subsection{Existence}\label{ssec: SNC existence}

In this subsection we prove Proposition \ref{prop: SNC exist}. 
In view of Proposition \ref{prop: SNC condition}, it suffices to prove the following. 

\begin{proposition}\label{prop: SNC cone exist}
For any ${\G}$-admissible collection of fans $\Sigma$, 
there exists a smooth projective ${\G}$-admissible subdivision of $\Sigma$ 
satisfying Condition \ref{separate condition}. 
\end{proposition}

For the proof, we first note the following. 

\begin{lemma}\label{lem: separate subdivision}
If a ${\G}$-admissible collection of fans $\Sigma'$ satisfies Condition \ref{separate condition}, 
any ${\G}$-admissible subdivision $\Sigma''\prec \Sigma'$ also satisfies Condition \ref{separate condition}. 
\end{lemma}

\begin{proof}
Let $\sigma''$ be a cone in $\Sigma''_{F}$ and 
$\tau_{+}''\ne \tau_{-}''$ be two rays of $\sigma''$. 
Let $\tau_{+}', \tau_{-}'$ be the minimal cones in $\Sigma'$ with $\tau_{\pm}''\subset \tau_{\pm}'$. 
($\tau_{\pm}'$ belongs to the same cusp as $\tau_{\pm}''$.) 
We take a cone $\sigma'\in \Sigma_{F}'$ such that $\sigma''\subset \sigma'$. 
By the minimality, we see that $\tau_{+}', \tau_{-}'$ are faces of $\sigma'$. 
Now, if $\gamma\tau_{+}''=\tau_{-}''$ for some $\gamma\in {\G}$, 
then $\gamma\tau_{+}'=\tau_{-}'$ by the minimality. 
Since $\gamma\colon \tau_{+}' \to \tau_{-}'$ is not the identity, 
this shows that two rays of $\sigma'$ are ${\G}$-equivalent, which is absurd. 
\end{proof}

Now we give the proof of Proposition \ref{prop: SNC cone exist}. 

\begin{proof}[(Proof of Proposition \ref{prop: SNC cone exist})]
We may start with a smooth $\Sigma$. 
We will take a ${\G}$-admissible subdivision $\Sigma' \prec \Sigma$ 
which satisfies Condition \ref{separate condition} (but not smooth). 
By Lemma \ref{lem: separate subdivision}, 
it then suffices to take a smooth projective ${\G}$-admissible subdivision of $\Sigma'$ following \cite{AMRT}. 

We consider the set $\mathcal{S}$ of ${\G}$-equivalence classes of $2$-dimensional cones in $\Sigma$. 
For each class $[\sigma]\in \mathcal{S}$, 
we choose a two-division of a representative $2$-dimensional cone $\sigma$. 
By letting ${\G}$ act and doing this for all classes in $\mathcal{S}$, 
we obtain a ${\G}$-invariant collection of two-divisions of all $2$-dimensional cones in $\Sigma$. 
Note that this is well-defined, i.e., every $2$-dimensional cone is two-divided exactly once, 
because the stabilizer of a $2$-dimensional cone acts trivially on this cone by the neatness of ${\G}$. 
Now, since $\Sigma$ is simplicial, 
each two-division of a $2$-dimensional cone $\sigma$ 
induces a two-division of all cones containing $\sigma$ uniquely and compatibly: 
thus a codimension $1$ wall is added inside each top-dimensional cone containing $\sigma$. 
(This is essentially a so-called star subdivision.) 
In this way we obtain a ${\G}$-admissible subdivision $\Sigma'$ of $\Sigma$. 
(See Figure \ref{figure: Sigma'}.) 

We shall show that $\Sigma'$ satisfies Condition \ref{separate condition}. 
Let $\tau_{+}'\ne \tau_{-}'$ be two rays in $\Sigma'$ 
which are contained in a common cone $\sigma'\in \Sigma_{F}'$. 
We want to show that $\tau_{+}'$ and $\tau_{-}'$ are not ${\G}$-equivalent. 
As in the proof of Lemma \ref{lem: separate subdivision}, 
we take a cone $\sigma\in \Sigma_{F}$ with $\sigma'\subset \sigma$ and  
the minimal cones $\tau_{\pm}\in \Sigma$ with $\tau_{\pm}'\subset \tau_{\pm}$. 
Then $\tau_{+}, \tau_{-}$ are faces of $\sigma$. 
By the minimality, we have ${\rm int}(\tau_{\pm}')\subset {\rm int}(\tau_{\pm})$ 
where ${\rm int}$ stands for the interior. 

Now suppose to the contrary that $\gamma\tau_{+}'=\tau_{-}'$ for some $\gamma\in {\G}$. 
Then $\gamma\tau_{+}=\tau_{-}$ by the minimality of $\tau_{\pm}$. 
This implies $\tau_{+}\ne \tau_{-}$, 
because the stabilizer of a cone in $\Sigma$ acts trivially on this cone by the neatness of ${\G}$. 
Therefore $\tau_{+}$ and $\tau_{-}$ are two different faces of a common old cone $\sigma\in \Sigma_{F}$ 
such that the subsets ${\rm int}(\tau_{+}')$, ${\rm int}(\tau_{-}')$ of 
their interiors are still contained in a common new cone $\sigma'\in \Sigma_{F}'$. 
However, by the construction of our subdivision, 
the interiors of two equidimensional faces of a cone in $\Sigma_{F}$ 
must be separated by a codimension $1$ wall, so this cannot happen. 
This finishes the proof of Proposition \ref{prop: SNC cone exist}. 
\end{proof}

\begin{figure}[h]
\includegraphics[height=40mm, width=80mm]{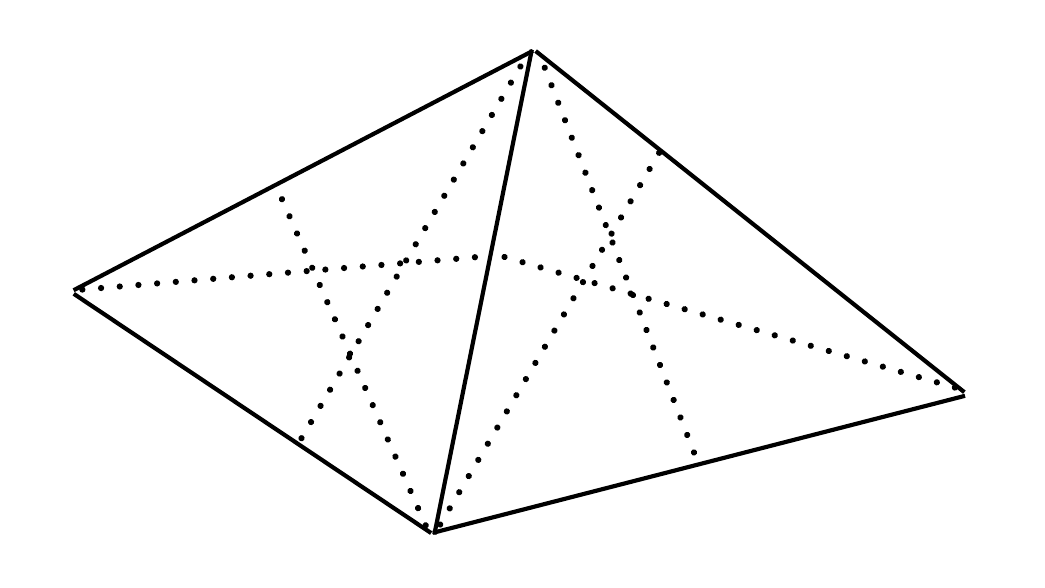}
\caption{$\Sigma'\prec \Sigma$}\label{figure: Sigma'}
\end{figure}





\end{document}